\documentclass[11pt,a4paper,oneside,reqno]{amsart}

\usepackage{latexsym}
\usepackage{amsmath}
\usepackage{amsthm}
\usepackage{amssymb}
\usepackage{amsfonts}
\usepackage{fancyhdr}
\usepackage{comment}

\usepackage{mathrsfs}                           

\newcommand{\mR}{\mathbf{R}}                    
\newcommand{\mC}{\mathbf{C}}                    
\newcommand{\abs}[1]{\lvert #1 \rvert}          
\newcommand{\norm}[1]{\lVert #1 \rVert}         
\newcommand{\br}[1]{\langle #1 \rangle}         

\newcommand{\eps}{\varepsilon}
\newcommand{\tphi}{\tilde{\varphi}}
\newcommand{\hphi}{\hat{\varphi}}
\newcommand{\vphi}{\varphi}

\newcommand{\re}{\mathrm{Re}}
\newcommand{\im}{\mathrm{Im}}
\newcommand{\supp}{\mathrm{supp}}

\newcommand{\closure}[1]{\overline{#1}}

\newcommand{\mOp}{\mathrm{Op}}

\theoremstyle{definition}
\newtheorem{thm}{Theorem}[section]
\newtheorem{prop}[thm]{Proposition}

\newtheorem{lemma}[thm]{Lemma}

\newtheorem*{remark}{Remark}

\numberwithin{equation}{section}

\title[Inverse problems with partial data for a Dirac system]{Inverse problems with partial data for a Dirac system: a Carleman estimate approach}
\author{Mikko Salo and Leo Tzou}
\address{Department of Mathematics and Statistics \\ University of Helsinki}
\email{mikko.salo@helsinki.fi}
\address{Department of Mathematics \\ Stanford University}
\email{leo.tzou@gmail.com}

\thanks{M.S.~is supported by the Academy of Finland.}
\thanks{L.T.~is supported by NSF grant DMS-0807502.}

\begin{document}

\begin{abstract}
We prove that the material parameters in a Dirac system with magnetic and electric potentials are uniquely determined by measurements made on a possibly small subset of the boundary. The proof is based on a combination of Carleman estimates for first and second order systems, and involves a reduction of the boundary measurements to the second order case. For this reduction a certain amount of decoupling is required. To effectively make use of the decoupling, the Carleman estimates are established for coefficients which may become singular in the asymptotic limit.
\end{abstract}

\maketitle

\section{Introduction}

This article is concerned with the inverse problem of determining unknown coefficients in a Dirac system from measurements made on part of the boundary. A standard problem of this type is the inverse conductivity problem of Calder{\'o}n \cite{calderon}, where the purpose is to determine the electrical conductivity of a body by making voltage to current measurements on the boundary. In mathematical terms, if $\gamma$ is a smooth positive function in the closure of a bounded domain $\Omega \subseteq \mR^n$, the boundary measurements are given by the Cauchy data set 
\begin{equation*}
C_{\gamma} = \{ (u|_{\partial \Omega}, \gamma \partial_{\nu} u|_{\partial \Omega}) \,;\, \nabla \cdot (\gamma \nabla u) = 0 \text{ in } \Omega, \ u \in H^1(\Omega) \}.
\end{equation*}
Here $u|_{\partial \Omega}$ and $\gamma \partial_{\nu} u|_{\partial \Omega}$ are the voltage and current, respectively, on $\partial \Omega$, corresponding to a potential $u$ satisfying the conductivity equation in $\Omega$ ($\partial_{\nu} u$ denotes the normal derivative). The inverse problem is to determine the conductivity $\gamma$ from the knowledge of the Cauchy data set $C_{\gamma}$.

The inverse conductivity problem has been well studied, and major results include \cite{astalapaivarinta}, \cite{nachman2d}, \cite{sylvesteruhlmann} which prove that $C_{\gamma}$ determines $\gamma$ in various settings. Less is known about the partial data problem, where one is given two sets $\Gamma_1, \Gamma_2 \subseteq \partial \Omega$ and the boundary measurements are encoded by the set 
\begin{equation*}
C_{\gamma}^{\Gamma_1,\Gamma_2} = \{ (u|_{\Gamma_1}, \gamma \partial_{\nu} u|_{\Gamma_2}) \,;\, \nabla \cdot (\gamma \nabla u) = 0 \text{ in } \Omega, \ u \in H^1(\Omega) \}.
\end{equation*}

There are two main approaches for proving that $\gamma$ is determined by $C_{\gamma}^{\Gamma_1,\Gamma_2}$. The first approach, introduced in \cite{bukhgeimuhlmann} and \cite{ksu}, uses Carleman estimates with boundary terms to control solutions on parts of the boundary. The result in \cite{ksu} is valid in dimensions $n \geq 3$ and for small sets $\Gamma_2$ (the shape depending on the geometry of $\partial \Omega$), but assumes that $\Gamma_1$ has to be relatively large. The second approach \cite{isakov} is based on reflection arguments and is valid when $n \geq 3$ and $\Gamma_1 = \Gamma_2$ and $\Gamma_1$ may be a small set, but it is limited to the case where $\partial \Omega \smallsetminus \Gamma_1$ is part of a hyperplane or a sphere. If $n=2$, a result similar to \cite{isakov} but without the last restriction was recently proved in \cite{iuy}.

We are interested in inverse problems with partial data for elliptic linear systems. In the case of full data (that is, $\Gamma_1 = \Gamma_2 = \partial \Omega$), there is an extensive literature including uniqueness results for the Maxwell equations \cite{ops}, \cite{os}, the Dirac system \cite{nakamuratsuchida}, \cite{salotzou}, and the elasticity system \cite{eskinralston_elasticity}, \cite{nakamurauhlmann}, \cite{nakamurauhlmannerratum}. However, it seems that partial data results for systems are more difficult to establish. The reflection approach is in principle more straightforward to extend to systems, and the recent work \cite{cos} gives a partial data result analogous to \cite{isakov} for the Maxwell equations. As for the Carleman estimate approach, there is a fundamental problem since Carleman estimates for first order systems, such as the ones in \cite{salotzou}, seem to have boundary terms which are not useful in partial data results.

In this paper, we prove a partial data result analogous to \cite{ksu} for a Dirac system. To our knowledge this is the first such partial data result for a system. The proof is based on Carleman estimates, and it involves a reduction to boundary measurements for a second order equation. The corresponding boundary term is handled by a Carleman estimate for second order systems, designed to take into account the amount of decoupling present in the original equation. In the set where one cannot decouple, we need to use the first order structure as well. The Carleman estimates need to be valid for coefficients which may blow up in the asymptotic limit, in order to obtain sufficiently strong estimates for solutions on the boundary.

Let us now state the precise problem. We consider the free Dirac operator in $\mR^3$, arising in quantum mechanics and given by the $4 \times 4$ matrix 
\begin{equation} \label{paulidirac_definition}
P(D) = \left( \begin{array}{cc} 0 & \sigma \cdot D \\ \sigma \cdot D & 0 \end{array} \right),
\end{equation}
where $D = -i\nabla$ and $\sigma = (\sigma_1, \sigma_2, \sigma_3)$ is a vector of Pauli matrices with 
\begin{equation*}
\sigma_1 = \left( \begin{array}{cc} 0 & 1 \\ 1 & 0 \end{array} \right), \ \sigma_2 = \left( \begin{array}{cc} 0 & -i \\ i & 0 \end{array} \right), \ \sigma_3 = \left( \begin{array}{cc} 1 & 0 \\ 0 & -1 \end{array} \right).
\end{equation*}
Let $\Omega \subseteq \mR^3$ be a bounded simply connected domain with $C^{\infty}$ boundary, let $A \in C^{\infty}(\closure{\Omega} ; \mR^3)$ be a vector field (magnetic potential), and let $q_{\pm}$ be two functions in $C^{\infty}(\closure{\Omega} ; \mR)$ (electric potentials). We will study a boundary value problem for the Dirac operator 
\begin{equation} \label{lv_definition}
\mathcal{L}_V = P(D) + V,
\end{equation}
where the potential $V$ has the form 
\begin{equation} \label{v_definition}
V = P(A) + Q = \left( \begin{array}{cc} q_+ I_2 & \sigma \cdot A \\ \sigma \cdot A & q_- I_2\end{array} \right),
\end{equation}
with $Q = \left( \begin{smallmatrix} q_+ I_2 & 0 \\ 0 & q_- I_2 \end{smallmatrix} \right)$.

Let $u$ be a $4$-vector $u = \left( \begin{smallmatrix} u_+ \\ u_- \end{smallmatrix} \right)$ where $u_{\pm} \in L^2(\Omega)^2$. By \cite[Section 4]{nakamuratsuchida}, the boundary value problem 
\begin{equation*}
\left\{ \begin{array}{rll}
\mathcal{L}_V u &\!\!\!= 0 & \quad \text{in } \Omega, \\
u_+ &\!\!\!= f & \quad \text{on } \partial \Omega,
\end{array} \right.
\end{equation*}
is well posed if $0$ is in the resolvent set of $\mathcal{L}_V$, and then there is a unique solution $u \in H^1(\Omega)^4$ for any $f \in H^{1/2}(\partial \Omega)^2$. The boundary measurements are given by the Dirichlet-to-Dirichlet map 
\begin{equation*} 
\Lambda_V: H^{1/2}(\partial \Omega) \to H^{1/2}(\partial \Omega), \ f \mapsto u_-|_{\partial \Omega}.
\end{equation*}
It is known that the map $\Lambda_V$ is preserved under a gauge transformation where $A$ is replaced by $A + \nabla p$ where $p|_{\partial \Omega} = 0$. Such a transformation does not change the magnetic field $\nabla \times A$, and the inverse problem is to recover the quantities $\nabla \times A$ and $q_{\pm}$ from the boundary measurements.

We are interested in the inverse problem with partial data, where the boundary information is the map $\Lambda_V$ restricted to a subset $\Gamma \subseteq \partial \Omega$. More generally, 
we can consider boundary measurements given by the restricted Cauchy data set 
\begin{equation*}
C_V^{\Gamma} = \{ (u_+|_{\partial \Omega}, u_-|_{\Gamma}) \,;\, u \in H^1(\Omega)^4 \text{ is a solution of } \mathcal{L}_V u = 0 \text{ in } \Omega \}.
\end{equation*}
If $0$ is in the resolvent set of $\mathcal{L}_V$, then $C_V^{\Gamma} = \{ (f, \Lambda_V f|_{\Gamma}) \,;\, f \in H^{1/2}(\partial \Omega)^2 \}$. Again, the set $C_V^{\Gamma}$ is preserved when $A$ is replaced by $A + \nabla p$ where $p|_{\partial \Omega} = 0$, so the inverse problem is to determine $\nabla \times A$ and $q_{\pm}$ from $C_V^{\Gamma}$.

We will prove the following partial data result. Let $\text{ch}(\closure{\Omega})$ be the convex hull of $\closure{\Omega}$, and if $x_0 \in \mR^3$ define the front face of $\partial \Omega$ by 
\begin{equation*}
F(x_0) = \{ x \in \partial \Omega \,;\, (x-x_0) \cdot \nu(x) \leq 0 \}.
\end{equation*}
If $\Gamma \subseteq \partial \Omega$, we write $\Gamma^c = \partial \Omega \smallsetminus \Gamma$ for the complement in $\partial \Omega$.

\begin{thm} \label{thm:uniqueness}
Let $\Omega \subseteq \mR^3$ be a bounded simply connected domain with connected $C^{\infty}$ boundary, let $A_1, A_2 \in C^{\infty}(\closure{\Omega} ; \mR^3)$, and let $q_{1,\pm}, q_{2,\pm} \in C^{\infty}(\closure{\Omega} ; \mR)$. Let $\Gamma$ be any neighborhood of $F(x_0)$ in $\partial \Omega$, where $x_0 \notin \text{ch}(\closure{\Omega})$, and assume the boundary conditions 
\begin{eqnarray}
 & A_1 = A_2 \quad \text{on } \partial \Omega, & \label{boundarycond1} \\
 & q_{1,\pm} = q_{2,\pm} \text{ and } \partial_{\nu} q_{1,\pm} = \partial_{\nu} q_{2,\pm} \text{ on } \partial \Omega, & \label{boundarycond2} \\
 & q_{1,-} \neq 0 \quad \text{on } \Gamma^c. & \label{boundarycond3}
\end{eqnarray}
If $C_{V_1}^{\Gamma} = C_{V_2}^{\Gamma}$, then $\nabla \times A_1 = \nabla \times A_2$ and $q_{1,\pm} = q_{2,\pm}$ in $\Omega$.
\end{thm}

In the full data case (when $\Gamma = \partial \Omega$), the inverse boundary problem for the Dirac system and the related fixed frequency inverse scattering problem have been considered in \cite{gotodirac}, \cite{isozakidirac}, \cite{lidirac}, \cite{nakamuratsuchida}, \cite{salotzou}, \cite{tsuchida}. In particular, Theorem \ref{thm:uniqueness} for full data was proved in \cite{nakamuratsuchida} for smooth coefficients and in \cite{salotzou} for Lipschitz continuous coefficients. For $\Gamma = \partial \Omega$ the boundary conditions \eqref{boundarycond1}--\eqref{boundarycond3} are not required, but for partial data results based on Carleman estimates as in \cite{bukhgeimuhlmann}, \cite{ksu} such conditions are usually needed at least on the inaccessible part $\Gamma^c$. By suitable boundary determination results and gauge transformations as in \cite{salotzou}, we expect that it would be enough to assume \eqref{boundarycond1} only for the tangential components of $A_1$ and $A_2$ on $\Gamma^c$ and \eqref{boundarycond2} only on $\Gamma^c$.

The most interesting condition is \eqref{boundarycond3}, which allows to decouple the Dirac system at least on some neighborhood of the inaccessible part $\Gamma^c$. This decoupling is required for the reduction from boundary measurements for Dirac to boundary measurements for a second order system, and also in patching the Carleman estimates for first and second order systems together to obtain decay for solutions on part of the boundary.

Let us outline the structure of the proof. In Section \ref{sec:integral_identity}, it is shown that the assumption $C_{V_1}^{\Gamma} = C_{V_2}^{\Gamma}$ along with \eqref{boundarycond1}--\eqref{boundarycond3} implies the integral identity
\begin{equation} \label{intro_identity}
\int_{\Omega} U_2^* (V_1-V_2) U_1 \,dx = -\int_{\Gamma^c} \frac{1}{q_{1,-}} U_{2,+}^* \partial_{\nu} U_+ \,dS
\end{equation}
where $U_1$ and $U_2$ are any $4 \times 4$ matrix solutions of $\mathcal{L}_{V_j} U_j = 0$ in $\Omega$, and further $U = U_1 - \tilde{U}_2$ where $\tilde{U}_2$ is a solution of $\mathcal{L}_{V_2} \tilde{U}_2 = 0$ in $\Omega$ with $\tilde{U}_{2,+}|_{\partial \Omega} = U_{1,+}|_{\partial \Omega}$. The normal derivative $\partial_{\nu} U_+$ corresponds to boundary measurements for a second order equation.

The matrices $U_1$ and $U_2$ will be complex geometrical optics solutions to the Dirac equation, depending on a small parameter $h$ and having logarithmic Carleman weights as phase functions. Such solutions were constructed for the Schr\"odinger equation in \cite{ksu} and for the Dirac equation in \cite{salotzou}. The construction relevant to this paper is presented in Section \ref{sec:cgo}.

The recovery of coefficients is given in Section \ref{sec:uniqueness}, and proceeds by inserting the complex geometrical optics solutions $U_1$ and $U_2$ into \eqref{intro_identity} and by letting $h \to 0$. With suitable choices, on the left hand side one obtains (nonlinear) two-plane transforms of the parameters involved, and microlocal analytic methods allow to determine the coefficients. The argument is an analog for the Dirac operator of results in \cite{dksu}, and also involves ideas from \cite{nakamuratsuchida}, \cite{salotzou}.

The remaining issue, and also the main contribution of this paper, is the analysis in terms of decay in $h$ of different parts of the boundary term in \eqref{intro_identity}. This is done in Section \ref{sec:carleman}. By a Carleman estimate, we may estimate $\partial_{\nu} U_+$ by a second order operator applied to $U_+$. We will apply an $h$-dependent decomposition of $\Omega$ into a set where $q_{2,-}$ is not too small (so one can decouple) and where $q_{2,-}$ is small, and the second order operator will be chosen accordingly. The coefficients of this operator will typically blow up when $h$ becomes very small.

The second order Carleman estimate is given for a phase function which is convexified by a parameter $\eps$ as in \cite{dksu} and \cite{ksu}, but there is the new feature that $\eps$ needs to depend on $h$ in a precise manner related to the decomposition of $\Omega$ to obtain sufficiently strong control of constants in the estimate. In the set where $q_{2,-}$ is small, we also use a Carleman estimate for the Dirac operator to obtain the final bounds.

More precisely, the proof of Theorem \ref{thm:uniqueness} proceeds in several steps. Noting that \eqref{intro_identity} is an identity for $4 \times 4$ matrices, the proof begins by looking at the upper right $2 \times 2$ blocks in \eqref{intro_identity} and by showing that $\nabla \times A_1 = \nabla \times A_2$. After a gauge transformation one may assume that $A_1 = A_2$, and then from the upper left and right $2 \times 2$ blocks of \eqref{intro_identity} one obtains that $q_{1,-} = q_{2,-}$, and also $q_{1,+} = q_{2,+}$ at all points where $q_{1,-}$ is nonzero. The coefficients $q_+$ would be recovered from the lower right $2 \times 2$ block of the integral identity, but the estimates for this block in the boundary term seem to be difficult. However, at this point one has enough information on the coefficients to go back to the Dirac equation and use unique continuation, so that the partial data problem can be reduced to the full data problem. Then the result of \cite{nakamuratsuchida} shows that $q_{1,+} = q_{2,+}$ everywhere, which ends the proof.

Finally, we remark that there is a large literature on Carleman estimates and unique continuation, also involving logarithmic weights. We refer to \cite{berthier_dirac}, \cite{jerison_dirac}, \cite{mandache_dirac} for such results for Dirac operators. Inverse problems for Dirac operators in time domain are discussed in \cite{kurylevlassas}.

\section{Integral identity} \label{sec:integral_identity}

The following integral identity will be used to determine the coefficients. We write $(u|v) = \int_{\Omega} v^* u \,dx$, $\norm{u}^2 = (u|u)$, and $(u|v)_{\Sigma} = \int_{\Sigma} v^* u \,dS$ where $u$ and $v$ are vectors or matrices in $\overline{\Omega}$, and $\Sigma$ is a subset of $\partial \Omega$.

\begin{lemma} \label{lemma:integral_identity}
Assuming the conditions in Theorem \ref{thm:uniqueness}, one has the identity 
\begin{equation*}
((V_1-V_2)u_1|u_2) = -(\frac{1}{q_{1,-}} \partial_{\nu} u_+|u_{2,+})_{\Gamma^c}
\end{equation*}
for any solutions $u_j \in H^1(\Omega)^4$ of $(P(D)+V_j) u_j = 0$ in $\Omega$, where $u = u_1 - \tilde{u}_2$ is a function in $(H^2 \cap H^1_0(\Omega))^2 \times H^1(\Omega)^2$ satisfying $\partial_{\nu} u_+|_{\Gamma} = 0$, and $\tilde{u}_2 \in H^1(\Omega)^4$ is a solution of $(P(D)+V_2) \tilde{u}_2 = 0$ in $\Omega$ with $\tilde{u}_{2,+}|_{\partial \Omega} = u_{1,+}|_{\partial \Omega}$ and $\tilde{u}_{2,-}|_{\Gamma} = u_{1,-}|_{\Gamma}$.
\end{lemma}
\begin{proof}
Note that the existence of $\tilde{u}_2$ with the stated properties is ensured by the condition $C_{V_1}^{\Gamma} = C_{V_2}^{\Gamma}$. We first show that 
\begin{equation} \label{integralidentity_first}
((V_1-V_2)u_1|u_2) = i((\sigma \cdot \nu)(u_{1,-} - \tilde{u}_{2,-}) | u_{2,+})_{\Gamma^c}
\end{equation}
Since $(P(D)w_1|w_2) = (w_1|P(D)w_2) + \frac{1}{i} (P(\nu) w_1|w_2)_{\partial \Omega}$ and $V_2^* = V_2$, we have 
\begin{align*}
((V_1-V_2)u_1|u_2) &= -(P(D)u_1|u_2) + (u_1|P(D)u_2) \\
 &= i(P(\nu) u_1|u_2)_{\partial \Omega} \\
 &= i(P(\nu) (u_1 - \tilde{u}_2)|u_2)_{\partial \Omega} + i(P(\nu) \tilde{u}_2|u_2)_{\partial \Omega}.
\end{align*}
Then \eqref{integralidentity_first} follows since $(u_1 - \tilde{u}_2)_+|_{\partial \Omega} = 0$, $(u_1-\tilde{u}_2)_-|_{\Gamma} = 0$, and 
\begin{align*}
i(P(\nu) \tilde{u}_2|u_2)_{\partial \Omega} &= (\tilde{u}_2|P(D)u_2) - (P(D)\tilde{u}_2|u_2) \\
 &= (V_2 \tilde{u}_2|u_2) - (\tilde{u}_2|V_2 u_2) \\
 &= 0.
\end{align*}

Now $u \in H^1(\Omega)^4$ with $-\Delta u = P(D)(-V_1 u_1 + V_2 \tilde{u}_2) \in L^2(\Omega)^4$, and since $u_+ \in H^1_0(\Omega)^2$ we obtain $u_+ \in H^2(\Omega)^2$ by elliptic regularity. It remains to show that 
\begin{equation} \label{ddmap_normalderivative}
i q_{1,-} (\sigma \cdot \nu) u_- = -\partial_{\nu} u_+ \quad \text{on } \partial \Omega.
\end{equation}
Since $u_1$ and $\tilde{u}_2$ are solutions, we have 
\begin{align*}
\sigma \cdot (D+A_1) u_{1,+} + q_{1,-} u_{1,-} &= 0, \\
\sigma \cdot (D+A_2) \tilde{u}_{2,+} + q_{2,-} \tilde{u}_{2,-} &= 0.
\end{align*}
This shows that 
\begin{equation*}
q_{1,-} u_{1,-} - q_{2,-} \tilde{u}_{2,-} = -\sigma \cdot Du_+ - (\sigma \cdot A_1) u_{1,+} + (\sigma \cdot A_2) \tilde{u}_{2,+} \quad \text{in }\Omega.
\end{equation*}
Restricting to $\partial \Omega$ and using the boundary conditions on the coefficients, and writing $Du_+ = -i(\partial_{\nu} u_+)\nu + (Du_+)_{\text{tan}}$ on the boundary, where $A_{\text{tan}}$ is the tangential component of a vector field $A$, we obtain 
\begin{equation*}
q_{1,-} u_- = - \sigma \cdot Du_+ = i (\sigma \cdot \nu) \partial_{\nu} u_+ - \sigma \cdot (Du_+)_{\text{tan}} \quad \text{on } \partial \Omega.
\end{equation*}
Since $u_+ = 0$ on $\partial \Omega$ we have $(Du_+)_{\text{tan}} = 0$ on $\partial \Omega$, and \eqref{ddmap_normalderivative} follows upon multiplying the last identity by $i(\sigma \cdot \nu)$.
\end{proof}

\section{Construction of solutions} \label{sec:cgo}

The recovery of coefficients will proceed by inserting complex geometrical optics solutions $u_1$ and $u_2$ into the identity in Lemma \ref{lemma:integral_identity}. These solutions depend on a small parameter $h > 0$, and have the form $u = e^{-\rho/h} m$ where $\rho$ is a complex phase function and $m$ has an explicit form when $h \to 0$.

For second order elliptic equations, complex geometrical optics solutions go back to \cite{calderon}, \cite{sylvesteruhlmann} in the case where $\rho$ is a linear function, and they have been used extensively in inverse problems for different equations (see the surveys \cite{uhlmannicm}, \cite{uhlmannselecta}). A more general construction was presented in \cite{ksu}, allowing phase functions $\rho = \varphi + i\psi$ where $\varphi$ is a so called limiting Carleman weight and $\psi$ solves a related eikonal equation. See \cite{DKSaU} for a characterization of the limiting weights. In \cite{ksu}, the logarithmic weights $\varphi(x) = \log\,\abs{x-x_0}$ were used to obtain results in the inverse conductivity problem with partial data.

For the Dirac system considered in this article, a construction of complex geometrical optics solutions was given in \cite{salotzou}. This construction, specialized to logarithmic Carleman weights, will be reviewed here. Let $A, q_{\pm}$ be coefficients in $C^{\infty}(\closure{\Omega})$. Instead of $4$-vector solutions we will use $4 \times 4$ matrix solutions $U$ (so that every column of $U$ is a solution) to $\mathcal{L}_V U = 0$ in $\Omega$, having the form 
\begin{equation} \label{Usolutionform}
U = e^{-\rho/h}(C_0 + h C_1 + h^2 R).
\end{equation}
Here $h$ is a small parameter, $\rho = \varphi + i\psi$ is a complex phase function satisfying the eikonal equation $(\nabla \rho)^2 = 0$, $C_0$ and $C_1$ are smooth matrices with explicit form, and $R$ is a correction term.

We move to the specific choices of $\rho$ and $C_j$, following \cite{dksu} and \cite{salotzou}. Fix a point $x_0 \in \mR^3 \smallsetminus \closure{\text{ch}(\Omega)}$, where $\text{ch}(\Omega)$ is the convex hull of $\Omega$, and let $\varphi(x) = \log\,\abs{x-x_0}$. We choose 
\begin{equation*}
\psi(x) = \text{dist}_{S^2}\left( \frac{x-x_0}{\abs{x-x_0}}, \omega \right),
\end{equation*}
where $\omega \in S^2$ is chosen so that $\psi$ is smooth near $\closure{\Omega}$. Then $\rho = \varphi + i\psi$ satisfies $(\nabla \rho)^2 = 0$ near $\closure{\Omega}$.

It will be convenient to make a change of coordinates as in \cite{dksu}. Choose coordinates so that $x_0 = 0$, $\omega = e_1$, and $\closure{\Omega} \subseteq \{x_3 > 0\}$. Write $x = (x_1,r e^{i\theta})$ where $r > 0$ and $\theta \in (0,\pi)$, and introduce the complex variable $z = x_1 + ir$. Also write $e_r = (0,\cos \theta,\sin \theta)$ and $\zeta = e_1 + i e_r$. In these coordinates one has 
\begin{eqnarray*}
 & \rho = \log\,z, \quad \nabla \rho = \frac{1}{z} \zeta, \quad \Delta \rho = -\frac{2}{z(z-\bar{z})}. & 
\end{eqnarray*}

As in \cite[Section 3]{salotzou}, the matrices $C_0$ and $C_1$ will be chosen to satisfy transport equations involving the Cauchy operator $\zeta \cdot D$. We will also use a function $\phi \in C^{\infty}(\closure{\Omega})$ solving 
\begin{equation*}
\zeta \cdot (\nabla \phi + A) = 0 \quad \text{in } \Omega.
\end{equation*}
A particular solution $\phi$ is obtained by extending $A$ smoothly into $\mR^3$ as a compactly supported vector field, and by letting $\phi = (\zeta \cdot \nabla)^{-1}(-\zeta \cdot A)$ where the Cauchy transform is defined by 
\begin{equation*}
(\zeta \cdot \nabla)^{-1} f(x) = \frac{1}{2\pi} \int_{\mR^2} \frac{1}{y_1+i y_2} f(x-y_1 \re\,\zeta - y_2 \im\,\zeta) \,dy_1 \,dy_2.
\end{equation*}
Below, we will always understand that $\phi$ is this solution. The extension of $A$ outside $\Omega$ will not play any role in the final results.

The following proposition gives the existence and required properties for complex geometrical optics solutions. We use the notation introduced above, and the notation 
\begin{equation*}
Q_I = \left( \begin{array}{cc} q_+ I_2 &  \\  & q_- I_2 \end{array} \right)_I = \left( \begin{array}{cc} q_- I_2 &  \\  & q_+ I_2 \end{array} \right).
\end{equation*}
We also write $A \lesssim B$ to denote that $A \leq CB$ where $C$ is a constant which does not depend on $h$.

\begin{prop} \label{prop:cgo}
Let $a \in C^{\infty}(\closure{\Omega})$ satisfy $(\zeta \cdot \nabla) a = 0$ in $\Omega$. Then for $h > 0$ sufficiently small, there exists a solution to $\mathcal{L}_V U = 0$ in $\Omega$ of the form \eqref{Usolutionform} where $\rho = \log\,z$, 
\begin{equation*}
C_0 = \frac{1}{z} P(\zeta) r^{-1/2} e^{i\phi} a
\end{equation*}
with $\zeta \cdot (\nabla \phi + A) = 0$ in $\Omega$, and 
\begin{equation*}
C_1 = \frac{1}{i}(P(D+A)-Q_I)(r^{-1/2} e^{i\phi} a) + \frac{1}{z} P(\zeta) \tilde{C_1}
\end{equation*}
with $\norm{\tilde{C}_1}_{W^{1,\infty}(\Omega)} \lesssim 1$. Further, we have 
\begin{equation*}
\norm{R}_{H^1(\Omega)} \lesssim 1.
\end{equation*}
\end{prop}
\begin{proof}
To obtain the $H^1(\Omega)$ estimate for $R$, in fact we need to compute more terms in the asymptotic expansion in terms of $h$ and look for a solution of the form 
\begin{equation*}
U = e^{-\rho/h}(C_0 + h C_1 + h^2 C_2 + h^3 C_3 + h^3 R_4).
\end{equation*}
With the choices of smooth matrices $C_j$ given below, Proposition 3.1 in \cite{salotzou} implies the existence of such a solution with $\norm{C_j}_{W^{1,\infty}} \lesssim 1$ and $\norm{R_4}_{L^2(\Omega)} + h \norm{\nabla R_4}_{L^2(\Omega)} \lesssim 1$ if $h$ is small enough. We then obtain the required solution \eqref{Usolutionform} upon taking $R = C_2 + h C_3 + h R_4$.

The conditions for $C_j$ in \cite[Proposition 3.1]{salotzou} are 
\begin{eqnarray*}
 & C_0 = P(\nabla \rho) \tilde{C}_0, \quad M_A \tilde{C_0} = 0, & \\
 & C_1 = \frac{1}{i}(P(D+A)-Q_I) \tilde{C}_0 + P(\nabla \rho) \tilde{C}_1, \quad M_A \tilde{C}_1 = i H_{A,W} \tilde{C_0}, & \\
 & C_2 = \frac{1}{i}(P(D+A)-Q_I) \tilde{C}_1 + P(\nabla \rho) \tilde{C}_2, \quad M_A \tilde{C}_2 = i H_{A,W} \tilde{C_1}, & \\
 & C_3 = \frac{1}{i}(P(D+A)-Q_I) \tilde{C}_2. & 
\end{eqnarray*}
Here $M_A$ and $H_{A,W}$ are the transport and Schr\"odinger operators 
\begin{eqnarray*}
 & M_A = (2 \nabla \rho \cdot (D+A) + \frac{1}{i} \Delta \rho) I_4, & \\
 & H_{A,W} = (D+A)^2 I_4 + \begin{pmatrix} \sigma \cdot (\nabla \times A) - q_+ q_- I_2 & -\sigma \cdot Dq_+ \\ -\sigma \cdot D q_- & \sigma \cdot (\nabla \times A) - q_+ q_- I_2 \end{pmatrix}. & 
\end{eqnarray*}
Also, $\tilde{C}_j$ are smooth matrices in $\closure{\Omega}$ solving the transport equations.

Let $\phi$ and $a$ be as stated. Using the special coordinates, we have 
\begin{equation*}
M_A = \frac{1}{z} \left( 2\zeta \cdot (D+A) + \frac{1}{r} \right) I_4.
\end{equation*}
Then $\tilde{C}_0 = r^{-1/2} e^{i\phi} a I_4$ solves $M_A \tilde{C}_0 = 0$ in $\Omega$, and $C_0$ has the desired form. Now one can solve the transport equations $\tilde{C}_2$ and $\tilde{C}_3$ by the Cauchy transform for instance, and this shows that also $C_1$ is as required.
\end{proof}

\begin{remark}
It is possible to perform the above construction of solutions with $\rho$ replaced by $-\rho$ or $\bar{\rho}$, since these functions also solve the eikonal equation. The corresponding forms for the solutions are, respectively, 
\begin{align*}
U &= e^{\rho/h} \Big[ -\frac{1}{z} P(\zeta) r^{-1/2} e^{i\phi} a + \frac{h}{i}(P(D+A)-Q_I)(r^{-1/2} e^{i\phi} a) \\
 & \qquad \qquad - \frac{h}{z} P(\zeta) \tilde{C_1} + O(h^2) \Big], \\
U &= e^{-\bar{\rho}/h} \Big[ \frac{1}{\bar{z}} P(\bar{\zeta}) r^{-1/2} e^{i\bar{\phi}} \bar{a} + \frac{h}{i}(P(D+A)-Q_I)(r^{-1/2} e^{i\bar{\phi}} \bar{a}) \\
 & \qquad \qquad + \frac{h}{\bar{z}} P(\bar{\zeta}) \tilde{C_1} + O(h^2) \Big],
\end{align*}
where $\zeta \cdot (\nabla \phi + A) = 0$ and $\zeta \cdot \nabla a = 0$ in $\Omega$, and $\norm{\tilde{C}_1}_{W^{1,\infty}(\Omega)} \lesssim 1$.
\end{remark}

\section{Uniqueness proof} \label{sec:uniqueness}

In this section, we give the proof of Theorem \ref{thm:uniqueness} modulo the estimates for boundary terms which are contained in Section \ref{sec:carleman}. The following simple algebraic identities, valid for $a, b \in \mC^3$, will be used many times in the computations below:
\begin{eqnarray*}
 & (\sigma \cdot a)(\sigma \cdot b) + (\sigma \cdot b)(\sigma \cdot a) = 2(a \cdot b) I_2, \quad (\sigma \cdot a)^2 = (a \cdot a) I_2, & \\
 & P(a)P(b) + P(b)P(a) = 2(a \cdot b) I_4, \quad P(a)^2 = (a \cdot a) I_4, & \\
 & P(a)Q = Q_I P(a). & 
\end{eqnarray*}
Since $\zeta \cdot \zeta = 0$, we also have $(\sigma \cdot \zeta)^2 = 0$ and $P(\zeta)^2 = 0$.

The starting point for the recovery of the coefficients is Lemma \ref{lemma:integral_identity}, which implies that 
\begin{equation} \label{matrix_integral_identity}
((V_1-V_2)U_1|U_2) = -(\frac{1}{q_{1,-}} \partial_{\nu} U_+|U_{2,+})_{\Gamma^c}
\end{equation}
where $U_j$ are $4 \times 4$ matrix solutions of $\mathcal{L}_{V_j} U_j = 0$ in $\Omega$, $U = U_1-\tilde{U}_2$, and $\tilde{U}_2$ solves $\mathcal{L}_{V_2} \tilde{U}_2 = 0$ in $\Omega$ with $U_+|_{\partial \Omega} = 0$, $U_-|_{\Gamma} = 0$.

We use Proposition \ref{prop:cgo}, or more precisely the remark after it, and choose solutions $U_1$ and $U_2$ with 
\begin{align*}
U_1 &= e^{\rho/h} \Big[ -\frac{1}{z} P(\zeta) r^{-1/2} e^{i\phi_1} a_1 + \tilde{R}_1 \Big], \\
U_2^* &= e^{-\rho/h} \Big[ \frac{1}{z} a_2 P(\zeta) r^{-1/2} e^{-i\phi_2} + \tilde{R}_2 \Big],
\end{align*}
where $\zeta \cdot (\nabla \phi_j + A_j) = 0$ and $\zeta \cdot \nabla a_j = 0$ in $\Omega$, and where $\norm{\tilde{R}_1}_{H^1(\Omega)} \lesssim h$ and $\norm{\tilde{R}_2}_{H^1(\Omega)} \lesssim h$.

The next result, whose proof is given in the next section, takes care of part of the boundary term in \eqref{matrix_integral_identity}.

\begin{lemma} \label{lemma:boundaryterm_magnetic}
The upper right $2 \times 2$ block of $(\frac{1}{q_{1,-}} \partial_{\nu} U_+|U_{2,+})_{\Gamma^c}$ is $o(1)$ as $h \to 0$.
\end{lemma}

It is now possible to show that the magnetic field is determined by partial boundary measurements.

\begin{lemma}
$\nabla \times A_1 = \nabla \times A_2$ in $\Omega$.
\end{lemma}
\begin{proof}
Since $P(\zeta) Q_j P(\zeta) = P(\zeta) P(\zeta) (Q_j)_I = 0$, the left hand side of \eqref{matrix_integral_identity}, with the above choices for $U_1$ and $U_2$, becomes 
\begin{equation*}
\int_{\Omega} U_2^* (V_1-V_2) U_1 \,dx = -\int_{\Omega} P(\zeta) P(A_1-A_2) P(\zeta)  \frac{e^{i(\phi_1-\phi_2)} a_1 a_2}{z^2 r} \,dx + O(h).
\end{equation*}
The identity $P(\zeta) P(A) = -P(A) P(\zeta) + 2(\zeta \cdot A) I_4$ implies 
\begin{equation*}
\int_{\Omega} U_2^* (V_1-V_2) U_1 \,dx = -2 \int_{\Omega} P(\zeta) (\zeta \cdot (A_1-A_2)) \frac{e^{i(\phi_1-\phi_2)} a_1 a_2}{z^2 r} \,dx + O(h).
\end{equation*}

Taking the limit as $h \to 0$ in the upper right $2 \times 2$ block of \eqref{matrix_integral_identity}, gives by Lemma \ref{lemma:boundaryterm_magnetic} that 
\begin{equation*}
\int_{\Omega} e^{i(\phi_1-\phi_2)} (\sigma \cdot \zeta)(\zeta \cdot (A_1-A_2)) a_1 a_2 z^{-2} r^{-1} \,dx = 0.
\end{equation*}
We choose $a_1(z,\theta) = z^2 g(z) b(\theta)$ and $a_2(z,\theta) = 1$, where $g(z)$ is a holomorphic and smooth function in the closure of $\Omega_{\theta} = \{ z \in \mC \,;\, (x_1,r e^{i\theta}) \in \Omega\}$, and $b(\theta)$ is any smooth function. Note that $\zeta = \zeta(\theta)$. Moving to polar coordinates in the $x'$ variables and by varying $b(\theta)$, we obtain that for all $\theta$ 
\begin{equation*}
(\sigma \cdot \zeta) \int_{\Omega_{\theta}} e^{i(\phi_1-\phi_2)} (A_1-A_2) \cdot (e_1+i e_r) g(z) \,d\bar{z} \wedge dz = 0.
\end{equation*}
Since $\sigma \cdot \zeta$ is not zero for any $\theta$, it follows that 
\begin{equation*}
\int_{\Omega_{\theta}} e^{i(\phi_1-\phi_2)} (A_1-A_2) \cdot (e_1+i e_r) g(z) \,d\bar{z} \wedge dz = 0.
\end{equation*}

The last expression is related to a (nonlinear) two-plane transform of $\nabla \times (A_1-A_2)$ over a set of two-planes. We may now apply the arguments in \cite[Section 5]{dksu} (see also \cite[Section 7]{knudsensalo}, where the last identity is the same as formula (40)). One first shows by complex analytic methods that the identity remains true with $e^{i(\phi_1-\phi_2)}$ and $g$ replaced by $1$. It follows that 
\begin{equation*}
\int_{\Omega_{\theta}} \xi \cdot (A_1-A_2) \,d\bar{z} \wedge dz = 0
\end{equation*}
whenever $\xi$ is in the two-plane spanned by $e_1$ and $e_r$. Varying $x_0$ and $\omega$ in the construction of solutions slightly, this implies that 
\begin{equation*}
\int_{P \cap \Omega} \xi \cdot (A_1-A_2) \,dS = 0
\end{equation*}
for all two-planes $P$ such that the distance between the tangent space $T(P)$ and the point $(0,e_1)$ is small. Finally, an argument involving the microlocal Helgason support theorem and the microlocal Holmgren theorem shows that $\nabla \times A_1 = \nabla \times A_2$ in $\Omega$.
\end{proof}

Since $\Omega$ is simply connected and $\nabla \times A_1 = \nabla \times A_2$, we see that $A_1 - A_2 = \nabla p$ for some function $p \in C^{\infty}(\closure{\Omega})$. Also, by the assumption that $A_1 = A_2$ on $\partial \Omega$, we see that $p$ is constant on the connected set $\partial \Omega$. Thus, we can assume that $p|_{\partial \Omega} = 0$ by substracting a constant. Then $C_{V_2}^{\Gamma}$ is preserved under the gauge transformation $A_2 \mapsto A_2 + \nabla p$, and consequently we may assume that $A_1 \equiv A_2$. We shall write $A = A_1 = A_2$ and $\phi = \phi_1 = \phi_2$.

By Proposition \ref{prop:cgo} there exist solutions $U_1$ and $U_2$ to the equations $\mathcal{L}_{V_j} U_j = 0$ in $\Omega$ $(j=1,2)$, such that 
\begin{align*}
U_1 &= e^{\rho/h} e^{i\phi} \Big[ -\frac{1}{z} P(\zeta) r^{-1/2} a_1 \\
 & \quad + \frac{h}{i} (P(D+\nabla \phi+A)-Q_{1,I})(r^{-1/2} a_1) - \frac{h}{z} P(\zeta) \hat{C}_1 + h^2 \hat{R}_1 \Big], \\
U_2^* &= e^{-\rho/h} e^{-i\phi} \Big[ \frac{1}{z} P(\zeta) r^{-1/2} a_2 \\
 & \quad + \frac{h}{i} (P(D-\nabla \phi-A)+Q_{2,I})(r^{-1/2} a_2) + \frac{h}{z} \hat{C}_2^* P(\zeta) + h^2 \hat{R}_2 \Big],
\end{align*}
where $\zeta \cdot (\nabla \phi + A) = 0$, $\zeta \cdot \nabla a_j = 0$, $\norm{\hat{C}_j}_{W^{1,\infty}} \lesssim 1$, and $\norm{\hat{R}_j}_{H^1(\Omega)} \lesssim 1$. 

With these choices for $U_1$ and $U_2$, we have the following result for the boundary term in \eqref{matrix_integral_identity} which will be used in recovering the electric potentials. Again, the proof is deferred to the next section.

\begin{lemma} \label{lemma:boundaryterm_electric}
The upper left and right $2 \times 2$ blocks of $(\frac{1}{q_{1,-}} \partial_{\nu} U_+|U_{2,+})_{\Gamma^c}$ are $o(h)$ as $h \to 0$.
\end{lemma}

From the upper left and right $2 \times 2$ blocks of \eqref{matrix_integral_identity}, it turns out that one can recover $q_-$ everywhere and $q_+$ at those points where $q_- \neq 0$.

\begin{lemma}
One has $q_{1,-} = q_{2,-}$ in $\Omega$. Also, $q_{1,+} = q_{2,+}$ at each point of $\Omega$ where $q_{1,-}$ is nonzero.
\end{lemma}
\begin{proof}
We introduce the notations $\hat{Q} = Q_1-Q_2$ and $\tilde{a}_j = r^{-1/2} a_j$ to make the formular shorter. Now $V_1 - V_2 = \hat{Q}$, so \eqref{matrix_integral_identity} becomes 
\begin{multline} \label{identity_electric_potential}
\int_{\Omega} \Big[ \frac{1}{z} P(\zeta) \tilde{a}_2 + \frac{h}{i} (P(D-\nabla \phi-A)+Q_{2,I})\tilde{a}_2 + \frac{h}{z} \hat{C}_2^* P(\zeta) \\
 + h^2 \hat{R}_2 \Big] \hat{Q} \Big[ -\frac{1}{z} P(\zeta) \tilde{a}_1 + \frac{h}{i} (P(D+\nabla \phi+A)-Q_{1,I})\tilde{a}_1 \\
 - \frac{h}{z} P(\zeta) \hat{C}_1 + h^2 \hat{R}_1 \Big] \,dx = -(\frac{1}{q_{1,-}} \partial_{\nu} U_+|U_{2,+})_{\Gamma^c}.
\end{multline}
Since $P(\zeta)\hat{Q}P(\zeta) = P(\zeta)P(\zeta)\hat{Q}_I = 0$, the term on the left of \eqref{identity_electric_potential} which is $O(1)$ with respect to $h$ vanishes. Also, for similar reasons, all terms involving $\frac{h}{z} P(\zeta) \hat{C}_1$ and $\frac{h}{z} \hat{C}_2^* P(\zeta)$ and $\hat{R}_j$ behave like $O(h^2)$. Thus we obtain 
\begin{multline} \label{identity_electric_potential_2}
\frac{h}{i} \int_{\Omega} \frac{1}{z} \Big[ P(\zeta) \hat{Q} \big\{(P(D+\nabla \phi+A)-Q_{1,I})\tilde{a}_1\big\} \tilde{a}_2 \\
  - \big\{(P(D-\nabla \phi-A)+Q_{2,I})\tilde{a}_2\big\} \hat{Q} P(\zeta) \tilde{a}_1 \Big] \,dx + O(h^2) \\
 = -(\frac{1}{q_{1,-}} \partial_{\nu} U_+|U_{2,+})_{\Gamma^c}.
\end{multline}
Also the terms involving $\nabla \phi + A$ vanish because 
\begin{align*}
 &P(\zeta) \hat{Q} P(\nabla \phi + A) + P(\nabla \phi + A) \hat{Q} P(\zeta) \\
 &= [P(\zeta) P(\nabla \phi + A) + P(\nabla \phi + A) P(\zeta)] \hat{Q}_I \\
 &= 2 [\zeta \cdot (\nabla \phi + A)] \hat{Q}_I = 0.
\end{align*}
The expression \eqref{identity_electric_potential_2} becomes 
\begin{multline} \label{identity_electric_potential_3}
\frac{h}{i} \int_{\Omega} \frac{1}{z} \Big[ P(\zeta) \hat{Q} \big\{(P(D)-Q_{1,I})\tilde{a}_1\big\} \tilde{a}_2 - \big\{(P(D)+Q_{2,I})\tilde{a}_2\big\} \hat{Q} P(\zeta) \tilde{a}_1 \Big] \,dx \\
 + O(h^2) = -(\frac{1}{q_{1,-}} \partial_{\nu} U_+|U_{2,+})_{\Gamma^c}.
\end{multline}
Note that $-\hat{Q} Q_{1,I} - Q_2 \hat{Q}_I = \hat{q} I_4$ where $\hat{q} = q_{2,+} q_{2,-} - q_{1,+} q_{1,-}$, so \eqref{identity_electric_potential_3} can be written as 
\begin{multline} \label{identity_electric_potential_4}
\frac{h}{i} \int_{\Omega} \frac{1}{z} \Big[ P(\zeta) P(D\tilde{a}_1) \tilde{a}_2 - P(D\tilde{a}_2) P(\zeta) \tilde{a}_1 \Big] \hat{Q}_I \,dx \\
 + \frac{h}{i} \int_{\Omega} \frac{1}{z} P(\zeta) \hat{q} \tilde{a}_1 \tilde{a}_2 \,dx + O(h^2) = -(\frac{1}{q_{1,-}} \partial_{\nu} U_+|U_{2,+})_{\Gamma^c}.
\end{multline}
Now, in the second integral on the left of \eqref{identity_electric_potential_4}, the upper left $2 \times 2$ block is zero. Thus, multiplying \eqref{identity_electric_potential_4} by $h^{-1}$ and taking the limit as $h \to 0$ in the upper left $2 \times 2$ block, we obtain from Lemma \ref{lemma:boundaryterm_electric} that 
\begin{equation} \label{identity_electric_potential_5}
\int_{\Omega} \frac{1}{z} \Big[ (\sigma \cdot \zeta)(\sigma \cdot D\tilde{a}_1) \tilde{a}_2 - (\sigma \cdot D\tilde{a}_2)(\sigma \cdot \zeta) \tilde{a}_1 \Big] (q_{1,-} - q_{2,-}) \,dx = 0.
\end{equation}

At this point we make the choices 
\begin{equation*}
\tilde{a}_1 = r^{-1/2} z b_1(\theta), \quad \tilde{a}_2 = r^{-1/2},
\end{equation*}
where $b_1(\theta)$ is a smooth function. Since 
\begin{equation*}
\nabla \tilde{a}_1 = -\frac{1}{2} r^{-3/2} z b_1 e_r + r^{-1/2} \zeta b_1 + r^{-3/2} z \frac{\partial b_1}{\partial \theta} e_{\theta}
\end{equation*}
where $e_{\theta} = (0,-\sin \theta, \cos \theta)$, we have 
\begin{multline*}
\frac{1}{z} \Big[ (\sigma \cdot \zeta) (\sigma \cdot D\tilde{a}_1) \tilde{a}_2 - (\sigma \cdot D\tilde{a}_2) (\sigma \cdot \zeta) \tilde{a}_1 \Big] \\
 = \frac{1}{i} (\sigma \cdot \zeta) (\sigma \cdot e_{\theta}) r^{-2} \frac{\partial b_1}{\partial \theta} - \frac{1}{2i} \big[ (\sigma \cdot \zeta) (\sigma \cdot e_r) - (\sigma \cdot e_r) (\sigma \cdot \zeta) \big] r^{-2} b_1.
\end{multline*}
Using the identity 
\begin{equation*}
(\sigma \cdot a)(\sigma \cdot b) = (a \cdot b) I_2 + i \sigma \cdot (a \times b), \qquad a, b \in \mC^3,
\end{equation*}
we obtain from \eqref{identity_electric_potential_5} that 
\begin{equation} \label{identity_electric_potential_6}
\int_{\Omega} \Big[ i (\sigma \cdot \zeta) \frac{\partial b_1}{\partial \theta} - (\sigma \cdot e_{\theta}) b_1 \Big] (q_{1,-}-q_{2,-}) r^{-2} \,dx = 0.
\end{equation}

Using the condition \eqref{boundarycond2}, we may extend $q_{1,\pm}-q_{2,\pm}$ by zero outside $\Omega$ and therefore we can assume that $q_{1,\pm}-q_{2,\pm} \in C^1_c(\mR^3)$. We write \eqref{identity_electric_potential_6} as 
\begin{equation*}
\int_{0}^{\pi} \Big[ i(\sigma \cdot \zeta) \frac{\partial b_1}{\partial \theta} - (\sigma \cdot e_{\theta}) b_1 \Big] \Big( \int_{\mR} \int_0^{\infty} (q_{1,-}-q_{2,-})(x_1,r,\theta) r^{-1} \,dx_1 \,dr \Big) \,d\theta = 0.
\end{equation*}
Note that $\zeta$ and $e_{\theta}$ only depend on $\theta$. Integrating by parts in $\theta$ and using that $\frac{\partial \zeta}{\partial \theta} = i e_{\theta}$, we obtain 
\begin{equation*}
\int_{0}^{\pi} (\sigma \cdot \zeta) b_1 \Big( \int_{\mR} \int_0^{\infty} \frac{\partial (q_{1,-}-q_{2,-})}{\partial \theta}(x_1,r,\theta) r^{-1} \,dx_1 \,dr \Big) \,d\theta = 0.
\end{equation*}
Varying $b_1$, it follows that 
\begin{equation*}
(\sigma \cdot \zeta) \int_{\Omega_{\theta}} \frac{\partial(q_{1,-}-q_{2,-})}{\partial \theta}(x_1,r,\theta) r^{-1} \,d\bar{z} \wedge dz = 0,
\end{equation*}
for all $\theta$.

Since $\sigma \cdot \zeta$ is never zero, we finally get 
\begin{equation*}
\int_{\Omega_{\theta}} \frac{\partial (q_{1,-} - q_{2,-})}{\partial \theta} r^{-1} \,d\bar{z} \wedge dz = 0
\end{equation*}
for all $\theta$. This implies the vanishing of a Radon transform on certain planes. Now varying the point $x_0$ in the definition of $\varphi$, the direction $\omega \in S^2$ in the definition of of $\psi$, and varying $\theta$, we obtain from the microlocal Helgason and Holmgren theorems as in \cite{dksu} that 
\begin{equation*}
\frac{\partial (q_{1,-} - q_{2,-})}{\partial \theta} r^{-1} = 0 \quad \text{in } \Omega.
\end{equation*}
Thus $q_{1,-} - q_{2,-}$ is independent of $\theta$. Since $q_{1,-} - q_{2,-} \in C_c(\mR^3)$, we obtain $q_{1,-} = q_{2,-}$ in $\Omega$ as required.

Finally, we return to \eqref{identity_electric_potential_4} and now consider the upper right $2 \times 2$ block. In the first integral on the left this block is zero, so multiplying by $h^{-1}$ and letting $h \to 0$ in the upper right block gives by Lemma \ref{lemma:boundaryterm_electric} that 
\begin{equation*}
\int_{\Omega} \frac{1}{z} (\sigma \cdot \zeta) \hat{q} \tilde{a}_1 \tilde{a}_2 \,dx = 0.
\end{equation*}
By a similar argument as above, we obtain that $\hat{q} = 0$. Since $q_{1,-} = q_{2,-}$, this implies $q_{1,+} = q_{2,+}$ at each point where $q_{1,-}$ is nonzero.
\end{proof}

We have proved that $A_1 = A_2$ and $q_{1,-} = q_{2,-}$ in $\Omega$, and that $q_{1,+} = q_{2,+}$ at any point where $q_{1,-}$ is nonzero. The next logical step would be to consider the lower right $2 \times 2$ block of \eqref{matrix_integral_identity} to show that $q_{1,+} = q_{2,+}$ everywhere in $\Omega$. However, the estimates for the boundary term in this case appear to be quite difficult. We will choose another route and reduce the remaining step to the full data problem, by using unique continuation.

\begin{lemma}
Assume the conditions of Theorem \ref{thm:uniqueness}, and assume in addition that there is some neighborhood $W$ of $\Gamma^c$ in $\overline{\Omega}$ such that 
\begin{align*}
A_1 = A_2 = A &\quad \text{in } W, \\
q_{1,\pm} = q_{2,\pm} = q_{\pm} &\quad \text{in } W.
\end{align*}
Then $C_{V_1}^{\partial \Omega} = C_{V_2}^{\partial \Omega}$, that is, the boundary measurements with full boundary data coincide.
\end{lemma}
\begin{proof}
Without loss of generality, we assume that $\Gamma^c$ is connected (if not then argue on each connected piece). By shrinking $W$ if necessary, we may assume that also $W$ is connected and $q_- \neq 0$ in $W$.

Let $(f,g)$ be an element of $C_{V_1}^{\partial \Omega}$, so that there is a solution $u_1 \in H^1(\Omega)^4$ of $\mathcal{L}_{V_1} u_1 = 0$ in $\Omega$ such that $u_{1,+} = f$ and $u_{1,-} = g$ on $\partial \Omega$. Since $C_{V_1}^{\Gamma} = C_{V_2}^{\Gamma}$, there is a solution $u_2 \in H^1(\Omega)^4$ of $\mathcal{L}_{V_2} u_2 = 0$ in $\Omega$ satisfying 
\begin{equation*}
u_{1,+} = u_{2,+} \text{ on }\partial \Omega, \quad u_{1,-} = u_{2,-} \text{ on }\Gamma.
\end{equation*}
Set $u = u_1 - u_2$. Then clearly $u_+ = 0$ on $\partial \Omega$ and $u_- = 0$ on $\Gamma$, and by \eqref{ddmap_normalderivative} we also have $\partial_{\nu} u_+ = 0$ on $\Gamma$. Furthermore, since all coefficients are identical in $W$ and $q_- \neq 0$ in $W$, we have that $u_+$ satisfies 
\begin{equation*}
\left\{ \begin{array}{rll}
(-\Delta I_2 + 2(A \cdot D) I_2 - \frac{1}{q_-} (\sigma \cdot Dq_-) \sigma \cdot D + \tilde{Q})u_+ &\!\!\!= 0 & \quad \text{in } W, \\
u_+ = \partial_{\nu} u_+ &\!\!\!= 0 & \quad \text{on } W \cap \Gamma,
\end{array} \right.
\end{equation*}
where $\tilde{Q}$ is some smooth $2 \times 2$ matrix.

The last system has scalar principal part, and the unique continuation principle holds (this can be seen by applying a scalar Carleman estimate to both components of $u_+$, for details see \cite{KSaU}). Since $W$ is connected we conclude that $u_+ = 0$ in $W$, and consequently $\partial_{\nu} u_+$ vanishes on $\Gamma^c$. Since $q_- \neq 0$ in $W$, the relation \eqref{ddmap_normalderivative} again implies that $u_- = 0$ on all of $\partial \Omega$. We have proved that $(f,g) \in C_{V_2}^{\partial \Omega}$, showing that $C_{V_1}^{\partial \Omega} \subseteq C_{V_2}^{\partial \Omega}$. The inclusion $C_{V_2}^{\partial \Omega} \subseteq C_{V_1}^{\partial \Omega}$ is analogous.
\end{proof}

We have proved that all the conditions in the preceding lemma hold, so we obtain that $C_{V_1}^{\partial \Omega} = C_{V_2}^{\partial \Omega}$. The uniqueness result in \cite{nakamuratsuchida} (or \cite{salotzou}) for the full data case then implies that $q_{1,+} = q_{2,+}$ in $\Omega$. This ends the proof of Theorem \ref{thm:uniqueness}.

\section{Carleman estimates} \label{sec:carleman}

In this section we prove Carleman estimates and establish Lemmas \ref{lemma:boundaryterm_magnetic} and \ref{lemma:boundaryterm_electric} which allow to take care of the boundary term in the identity \eqref{matrix_integral_identity}. This involves an estimate for $\partial_{\nu} U_+$ on part of the boundary. To explain the strategy, we note that any solution $u$ of $\mathcal{L}_V u = 0$ in $\Omega$ satisfies 
\begin{align*}
\sigma \cdot (D+A) u_- + q_+ u_+ &= 0, \\
\sigma \cdot (D+A) u_+ + q_- u_- &= 0.
\end{align*}
Then in the set where $q_- \neq 0$, the equations decouple and we see that $u_+$ satisfies the second order equation 
\begin{equation} \label{carleman_secondorderequation}
\sigma \cdot (D+A) \left( \frac{1}{q_-} \sigma \cdot (D+A) u_+ \right) - q_+ u_+ = 0.
\end{equation}
We will estimate $\partial_{\nu} u_+$ on part of the boundary by using a Carleman estimate for a second order equation, as in \cite{ksu}. However, to account for the set where $q_-$ is small we need to do the analysis very carefully, cutting off the coefficients in a suitable $h$-dependent way and also letting the convexification depend on $h$. The details are given in the following result.

We will use below the notation given in the beginning of Section \ref{sec:integral_identity}, and also the sets $\partial \Omega_{\pm} = \{ x \in \partial \Omega \,;\, \pm \nabla \varphi(x) \cdot \nu(x) \geq 0 \}$. Further, we consider the semiclassical Sobolev spaces with norm defined by 
\begin{equation*}
\norm{u}_{H^s_{\text{scl}}} = \norm{(1+(hD)^2)^{s/2} u}_{L^2(\mR^n)}, \quad u \in C_c^{\infty}(\mR^n), \ s \in \mR.
\end{equation*}
In particular we will consider the case $s=1$ with the equivalent norm $\norm{u}_{H^1_{\text{scl}}} = \norm{u} + \norm{hDu}$ for $u \in H^1_0(\Omega)$.

\begin{lemma} \label{lemma:carleman_hdependent}
Let $A \in C^{\infty}(\closure{\Omega} \,;\, \mR^3)$ and $q_-, \tilde{q} \in C^{\infty}(\closure{\Omega})$, and let $\varphi$ be a limiting Carleman weight near $\closure{\Omega}$. Let $0 < \alpha < 1$, and let 
\begin{align*}
\hat{A}(x,D) &= \left\{ \begin{array}{cl} 2A \cdot D - \frac{1}{q_-}(\sigma \cdot Dq_-)\sigma \cdot D, & \left| \frac{1}{q_-} \right| \leq \sqrt{\abs{\log h^{\alpha}}}, \\
2A \cdot D, & \text{otherwise}, \end{array} \right. \\
\hat{q}(x) &= \left\{ \begin{array}{cl} - \frac{1}{q_-}(\sigma \cdot Dq_-)\sigma \cdot A + \tilde{q}, & \left| \frac{1}{q_-} \right| \leq \sqrt{\abs{\log h^{\alpha}}}, \\ \tilde{q}, & \text{otherwise}. \end{array} \right.
\end{align*}
There exist constants $h_0$, $C_0$, $C$, where $C_0$ and $C$ are independent of $\alpha$, such that whenever $0 < h \leq h_0$ and when 
\begin{equation*}
\tilde{\varphi} = \varphi + \frac{h}{\varepsilon(h)} \frac{\varphi^2}{2}, \qquad \varepsilon(h) = (C_0 \abs{\log h^{\alpha}})^{-1},
\end{equation*}
one has the estimate 
\begin{multline*}
\frac{h^2}{\varepsilon(h)} (\norm{e^{\tphi/h} v}^2 + \norm{e^{\tphi/h} hDv}^2) - h^3 \int_{\partial \Omega_-} \partial_{\nu} \vphi \abs{e^{\tphi/h} \partial_{\nu} v}^2 \,dS \\
 \leq C \norm{e^{\tphi/h} h^2 (-\Delta + \hat{A}(x,D) + \hat{q}) v}^2 + C h^3 \int_{\partial \Omega_+} \partial_{\nu} \vphi \abs{e^{\tphi/h} \partial_{\nu} v}^2 \,dS,
\end{multline*}
for any $v \in H^2(\Omega)$ with $v|_{\partial \Omega} = 0$.
\end{lemma}
\begin{proof}
The proof follows ideas in \cite{ksu} (see also \cite{salotzou}). First let $C_0$ be a fixed number, and initially choose $h_0 = 1$. Below we will replace $h_0$ by smaller constants when needed, and $M \geq 1$ will denote a changing constant depending only on $\varphi$ and the coefficients $A$, $q_-$, and $\tilde{q}$.

Write $\tphi = f(\varphi)$ where $f(\lambda) = \lambda + \frac{h}{\eps(h)} \frac{\lambda^2}{2}$, and introduce the conjugated operator $P_{0,\tphi} = e^{\tphi/h} (-h^2 \Delta) e^{-\tphi/h} = A + i B$ where $A$ and $B$ are the formally self-adjoint operators 
\begin{equation*}
A = (hD)^2 - (\nabla \tphi)^2, \quad B = \nabla \tphi \circ hD + hD \circ \nabla \tphi.
\end{equation*}
Then, if $v$ is as above, integration by parts gives that 
\begin{equation*}
\norm{P_{0,\tphi} v}^2 = \norm{Av}^2 + \norm{Bv}^2 + (i[A,B]v|v) - 2h^3 ((\partial_{\nu} \tphi) \partial_{\nu} v| \partial_{\nu} v)_{\partial \Omega}.
\end{equation*}
In terms of symbols one has $i[A,B] = h \mOp_h(\{a,b\})$. The limiting Carleman condition implies, as in \cite[Section 3]{ksu} and \cite[Lemma 2.1]{salotzou}, that 
\begin{equation*}
\{a,b\}(x,\xi) = \frac{4h}{\eps(h)} f'(\varphi)^2 \abs{\nabla \varphi}^4 + m(x) a(x,\xi) + l(x,\xi) b(x,\xi)
\end{equation*}
where 
\begin{equation*}
m(x) = -4 f'(\varphi) \frac{\varphi'' \nabla \varphi \cdot \nabla \varphi}{\abs{\nabla \varphi}^2}, \quad l(x,\xi) = \left( \frac{4\varphi'' \nabla \varphi}{\abs{\nabla \varphi}^2} + \frac{2 f''(\varphi)}{f'(\varphi)} \nabla \varphi \right) \cdot \xi.
\end{equation*}
Here we have chosen $h_0$ so that $\frac{h}{\eps(h)} \max(\sup\,\abs{\varphi},1) \leq 1/2$ for $h \leq h_0$, which ensures that $f'(\varphi) \geq 1/2$. Quantization gives 
\begin{equation*}
i[A,B] = \frac{4h^2}{\eps(h)} f'(\varphi)^2 \abs{\nabla \varphi}^4 + \frac{h}{2} \left[ m \circ A + A \circ m + L \circ B + B \circ L \right] + h^2 \tilde{q}(x)
\end{equation*}
where $\tilde{q}$ is a smooth function whose $C^k$ norms are uniformly bounded in $h$. Since $\abs{\nabla \varphi}$ is positive near $\closure{\Omega}$, we have 
\begin{multline*}
\norm{P_{0,\tphi} v}^2 \geq \norm{Av}^2 + \norm{Bv}^2 + \frac{h^2}{M \eps(h)} \norm{v}^2 \\
 - Mh \norm{v} \,\norm{Av} - Mh \norm{v}_{H^1_{\text{scl}}} \norm{Bv} - Mh^2 \norm{v}^2 - 2h^3 ((\partial_{\nu} \tphi) \partial_{\nu} v| \partial_{\nu} v)_{\partial \Omega}.
\end{multline*}
This used integration by parts and the fact that $v|_{\partial \Omega} = 0$. We obtain 
\begin{multline*}
\norm{P_{0,\tphi} v}^2 \geq \frac{1}{2} \norm{Av}^2 + \frac{1}{2} \norm{Bv}^2 + \frac{h^2}{M \eps(h)} \norm{v}^2 - M h^2 \norm{v}_{H^1_{\text{scl}}}^2 \\
 - 2h^3 ((\partial_{\nu} \tphi) \partial_{\nu} v| \partial_{\nu} v)_{\partial \Omega}.
\end{multline*}
For the term involving $\norm{v}_{H^1_{\text{scl}}}$ we note that 
\begin{equation*}
\norm{hDv}^2 = ((hD)^2 v|v) = (Av|v) + (\abs{\nabla \tphi}^2 v|v) \leq \norm{Av}^2 + M \norm{v}^2,
\end{equation*}
and by the assumptions on $h_0$ we have that 
\begin{equation*}
\frac{1}{2} \norm{Av}^2 \geq \frac{h^2}{2M^2 \eps(h)} \norm{Av}^2 \geq \frac{h^2}{2M^2 \eps(h)} \norm{hDv}^2 - \frac{h^2}{2 M \eps(h)} \norm{v}^2.
\end{equation*}
Therefore, the Carleman estimate becomes 
\begin{equation*}
\norm{P_{0,\tphi} v}^2 \geq \frac{h^2}{M \eps(h)} \norm{v}_{H^1_{\text{scl}}}^2 - 2h^3 ((\partial_{\nu} \tphi) \partial_{\nu} v| \partial_{\nu} v)_{\partial \Omega}.
\end{equation*}

Next, consider the operator 
\begin{align*}
P_{\tphi} &= h^2 e^{\tphi/h} (-\Delta + \hat{A}(x,D) + \hat{q}) e^{-\tphi/h} \\
 &= P_{0,\tphi} + h \hat{A}(x,hD+i\nabla \tphi) + h^2 \hat{q}.
\end{align*}
We have 
\begin{multline*}
\frac{h^2}{M \eps(h)} \norm{v}_{H^1_{\text{scl}}}^2 - 2h^3 ((\partial_{\nu} \tphi) \partial_{\nu} v| \partial_{\nu} v)_{\partial \Omega} \\
 \leq 4 \norm{P_{\tphi} v}^2 + 4 h^2 \norm{\hat{A}(x,hD+i\nabla \tphi)v}^2 + 4 h^4 \norm{\hat{q} v}^2.
\end{multline*}
If $h_0$ is chosen so that $\abs{\log(h^{\alpha})} \geq 1$ for $h \leq h_0$, then by the definition of $\hat{A}(x,D)$ and $\hat{q}(x)$ it holds that 
\begin{eqnarray*}
 & \norm{\hat{A}(x,hD+i\nabla \tphi)v} \leq M \sqrt{\abs{\log(h^{\alpha})}} \norm{v}_{H^1_{\text{scl}}}, & \\
 & \norm{\hat{q} v} \leq M \sqrt{\abs{\log(h^{\alpha})}} \norm{v}. & 
\end{eqnarray*}
Thus 
\begin{equation*}
\frac{h^2}{M \eps(h)} \norm{v}_{H^1_{\text{scl}}}^2 - 2h^3 ((\partial_{\nu} \tphi) \partial_{\nu} v| \partial_{\nu} v)_{\partial \Omega} \\
 \leq M \norm{P_{\tphi} v}^2 + M h^2 \abs{\log(h^{\alpha})} \norm{v}_{H^1_{\text{scl}}}^2.
\end{equation*}
At this point we choose $C_0$ so that $C_0 \geq 2 M^2$, which implies 
\begin{equation*}
M h^2 \abs{\log(h^{\alpha})} \leq \frac{h^2}{2 M \eps(h)}.
\end{equation*}
With this choice, we arrive at the Carleman estimate 
\begin{equation*}
\frac{h^2}{M \eps(h)} \norm{v}_{H^1_{\text{scl}}}^2 - 2h^3 ((\partial_{\nu} \tphi) \partial_{\nu} v| \partial_{\nu} v)_{\partial \Omega} \leq M \norm{P_{\tphi} v}^2.
\end{equation*}

Finally, replacing $v$ by $e^{\tphi/h} v$ gives that 
\begin{multline*}
\frac{h^2}{M \eps(h)} (\norm{e^{\tphi/h} v}^2 + \norm{hD(e^{\tphi/h} v)}^2) - 2h^3 ((\partial_{\nu} \tphi) e^{\tphi/h} \partial_{\nu} v| e^{\tphi/h} \partial_{\nu} v)_{\partial \Omega} \\
 \leq M \norm{e^{\tphi/h} h^2(-\Delta + \hat{A}(x,D) + \hat{q}) v}^2.
\end{multline*}
Since $\partial_{\nu} \tphi = f'(\vphi) \partial_{\nu} \vphi$, the result follows.
\end{proof}

\begin{remark}
The motivation for the choice of $\eps(h)$ comes from the fact that $e^{\tphi/h} \leq h^{-C\alpha} e^{\vphi/h}$ where $C$ is independent of $\alpha$. Here the factor $h^{-C\alpha}$ can be controlled by positive powers of $h$ if $\alpha$ is chosen small enough.
\end{remark}

It will be essential to use the convexified weight $\tphi$ instead of $\vphi$, since we will need the stronger constants obtained from convexification to carry out the estimates for boundary terms. Next we give a Carleman estimate for the Dirac operator, which will also be required for controlling the boundary terms.

\begin{lemma} \label{lemma:carleman_dirac}
Let $A \in C^{\infty}(\closure{\Omega} \,;\, \mR^3)$ and $q \in C^{\infty}(\closure{\Omega})$, and let $\varphi$ be a limiting Carleman weight near $\closure{\Omega}$. Suppose $\alpha$, $C_0$, $\eps(h)$, and $\tphi$ are as in Lemma \ref{lemma:carleman_hdependent}. There exist $C, h_0 > 0$, with $C$ independent of $\alpha$, such that for $0 < h \leq h_0$ one has 
\begin{equation*}
\norm{e^{\tphi/h} u}^2 \leq C \eps(h) \norm{ e^{\tphi/h} (\sigma \cdot (D+A) + q I_2) u}^2, \quad u \in H^1_0(\Omega)^2.
\end{equation*}
\end{lemma}
\begin{proof}
We follow the argument in \cite[Lemma 2.2]{salotzou}, where more details are given. The Carleman estimate in Lemma \ref{lemma:carleman_hdependent} implies that 
\begin{equation*}
\frac{h^2}{\varepsilon(h)} \norm{u}_{H^1_{\text{scl}}}^2 \leq C \norm{e^{\tphi/h} (-h^2 \Delta I_2) e^{\tphi/h} u}_{L^2}^2,
\end{equation*}
for all $u \in C^{\infty}_c(\Omega)^2$. It is possible to shift the estimate to a lower Sobolev index and prove that for $h$ small one has 
\begin{equation*}
\frac{h^2}{\varepsilon(h)} \norm{u}_{L^2}^2 \leq C \norm{e^{\tphi/h} (-h^2 \Delta I_2) e^{\tphi/h} u}_{H^{-1}_{\text{scl}}}^2.
\end{equation*}
Now write $e^{\tphi/h} (-h^2 \Delta I_2) e^{\tphi/h} = P_{\tphi}^2$ where $P_{\tphi}(hD) = e^{\tphi/h} (\sigma \cdot hD) e^{-\tphi/h}$. Since $\br{hD}^{-1} P_{\tphi}(hD)$ is an operator of order $0$, we obtain 
\begin{equation*}
\frac{h^2}{\varepsilon(h)} \norm{u}_{L^2}^2 \leq C \norm{e^{\tphi/h} (\sigma \cdot hD) e^{-\tphi/h} u}_{L^2}^2.
\end{equation*}
If $h$ is small (so $1/\eps(h)$ is large), we may replace $\sigma \cdot hD$ by $\sigma \cdot (hD + hA) + hq I_2$ in the last inequality. This shows the desired estimate for $u \in C^{\infty}_c(\Omega)^2$, and the result is then valid for $u \in H^1_0(\Omega)^2$ by approximation.
\end{proof}

We may now give the proof of Lemma \ref{lemma:boundaryterm_magnetic}, which provides an estimate for the part of the boundary term required for determining the magnetic field.

\begin{proof}[Proof of Lemma \ref{lemma:boundaryterm_magnetic}]
Recall that $U_j$ are solutions of $\mathcal{L}_{V_j} U_j = 0$ in $\Omega$, where 
\begin{align*}
U_{1,+} &= e^{\rho/h} \left[ -\frac{1}{z} r^{-1/2} e^{i\phi_1} a_1 \begin{pmatrix} 0 & \sigma \cdot \zeta \end{pmatrix}_{2 \times 4} + \begin{pmatrix} \hat{R}_1 & \hat{R}_1' \end{pmatrix}_{2 \times 4} \right], \\
U_{2,+}^* &= e^{-\rho/h} \Big[ \frac{1}{z} r^{-1/2} e^{-i\phi_2} a_2 \begin{pmatrix} 0 \\ \sigma \cdot \zeta \end{pmatrix}_{4 \times 2} + \begin{pmatrix} \hat{R}_2 \\ \hat{R}_2' \end{pmatrix}_{4 \times 2} \Big],
\end{align*}
with $\norm{\hat{R}_1}_{H^1(\Omega)} \lesssim h$, $\norm{\hat{R}_2}_{H^1(\Omega)} \lesssim h$. Also, $U = U_1 - \tilde{U}_2$ where $\tilde{U}_2$ solves $\mathcal{L}_{V_2} \tilde{U}_2 = 0$ in $\Omega$ with $\tilde{U}_{2,+}|_{\partial \Omega} = U_{1,+}|_{\partial \Omega}$. Thus $U_+|_{\partial \Omega} = 0$, and we have $U_-|_{\Gamma} = \partial_{\nu} U_+|_{\Gamma} = 0$ and also $U_+ \in (H^2 \cap H^1_0(\Omega))^{2 \times 4}$ by Lemma \ref{lemma:integral_identity}.

Denote by $J$ the upper right $2 \times 2$ block of $(\frac{1}{q_{1,-}} \partial_{\nu} U_+|U_{2,+})_{\Gamma^c}$. Writing $W = W_1 - \tilde{W}_2$ where $W_1$ and $\tilde{W}_2$ are the right $4 \times 2$ blocks of $U_1$ and $\tilde{U}_2$, respectively, we have 
\begin{align*}
J = \int_{\Gamma^c} e^{-\rho/h} \hat{R}_2 \frac{1}{q_{1,-}} \partial_{\nu} W_+ \,dS.
\end{align*}
Since $q_{1,-} \neq 0$ on $\Gamma^c$, we get (using the Frobenius norm on matrices) that 
\begin{equation*}
\norm{J}^2 \leq C \left( \int_{\Gamma^c} \norm{\hat{R}_2}^2 \,dx \right) \left( \int_{\Gamma^c} \norm{e^{-\varphi/h} \partial_{\nu} W_+}^2 \,dS \right).
\end{equation*}
Write $-\hphi = -\varphi + \frac{h}{\eps(h)} \frac{\varphi^2}{2}$ for the convexified weight corresponding to $-\varphi$, as in Lemma \ref{lemma:carleman_hdependent}. We note that $e^{-\varphi/h} = m e^{-\hphi/h}$ where $0 < m \leq 1$, and also the estimate $\norm{\hat{R}_2}_{L^2(\Gamma^c)} \leq C \norm{\hat{R}_2}_{H^1(\Omega)} \lesssim h$ which follows from the trace theorem. These facts yield 
\begin{equation} \label{first_j_estimate}
\norm{J}^2 \lesssim h^2 \norm{e^{-\hphi/h} \partial_{\nu} W_+}_{L^2(\Gamma^c)}^2.
\end{equation}

At this point we wish to use the Carleman estimate of Lemma \ref{lemma:carleman_hdependent}. This allows to estimate $\partial_{\nu} W_+$ by a second order operator applied to $W_+$. Since $W_+ = W_{1,+} - \tilde{W}_{2,+}$ where $W_1$ is a solution with explicit form which can be estimated, we want to choose the operator so that it will make terms involving $\tilde{W}_{2,+}$ vanish. Note that when $q_{2,-} \neq 0$, $\tilde{W}_{2,+}$ solves (columnwise) the equation 
\begin{equation*}
\sigma \cdot (D+A_2) \left( \frac{1}{q_{2,-}} \sigma \cdot (D+A_2) \tilde{W}_{2,+} \right) - q_{2,+} \tilde{W}_{2,+} = 0.
\end{equation*}
In the set where $q_{2,-} \neq 0$ this may be rewritten as 
\begin{equation} \label{w2tilde_equation}
[-\Delta I_2 + 2 (A_2 \cdot D) I_2 - \frac{1}{q_{2,-}} (\sigma \cdot Dq_{2,-}) \sigma \cdot (D+A_2) + \tilde{q}_2 I_2] \tilde{W}_{2,+} = 0,
\end{equation}
where $\tilde{q}_2 = A_2 \cdot A_2 + D \cdot A_2 - q_{2,+} q_{2,-}$. Since $W_+|_{\partial \Omega} = 0$, the estimate in Lemma \ref{lemma:carleman_hdependent}, applied to $-\varphi$ and $A_2$, $q_{2,-}$, and $\tilde{q}_2$, shows that 
\begin{multline*}
\frac{h^2}{\varepsilon(h)} (\norm{e^{-\hphi/h} W_+}^2 + \norm{e^{-\hphi/h} hDW_+}^2) + h^3 \int_{\partial \Omega_+} \partial_{\nu} \vphi \abs{e^{-\hphi/h} \partial_{\nu} W_+}^2 \,dS \\
 \lesssim \norm{e^{-\hphi/h} h^2 (-\Delta + \hat{A}_2(x,D) + \hat{q}_2) W_+}^2 - h^3 \int_{\partial \Omega_-} \partial_{\nu} \vphi \abs{e^{-\hphi/h} \partial_{\nu} W_+}^2 \,dS.
\end{multline*}
Recall that $\partial_{\nu} W_+|_{\Gamma} = 0$. Now $\Gamma$ is a neighborhood of the front face $F(x_0)$, but we have $F(x_0) = \partial \Omega_-$ since $\varphi$ was the logarithmic weight. This shows that the last boundary integral vanishes. Since $\partial_{\nu} \vphi > 0$ on $\Gamma^c$, the Carleman estimate can be written as 
\begin{multline*}
h^3 \int_{\Gamma^c} \abs{e^{-\hphi/h} \partial_{\nu} W_+}^2 \,dS + \frac{h^2}{\varepsilon(h)} (\norm{e^{-\hphi/h} W_+}^2 + \norm{e^{-\hphi/h} hDW_+}^2) \\
 \lesssim \norm{e^{-\hphi/h} h^2 (-\Delta + \hat{A}_2(x,D) + \hat{q}_2) W_+}^2.
\end{multline*}

Going back to \eqref{first_j_estimate}, we have arrived at 
\begin{multline} \label{j_carleman_estimate}
\norm{J}^2 + \frac{h}{\varepsilon(h)} (\norm{e^{-\hphi/h} W_+}^2 + \norm{e^{-\hphi/h} hDW_+}^2 ) \\
\lesssim h^3 \norm{e^{-\hphi/h} (-\Delta + \hat{A}_2(x,D) + \hat{q}_2) W_+}^2.
\end{multline}
Let $S_h = \{x \in \Omega \,;\, \left| \frac{1}{q_{2,-}} \right| \leq \sqrt{\abs{\log h^{\alpha}}} \}$ be a subset of $\Omega$ where $q_{2,-}$ is bounded away from zero, with an $h$-dependent bound. The proof will then be completed by establishing the following two estimates:
\begin{eqnarray}
 & h^3 \norm{e^{-\hphi/h} (-\Delta + \hat{A}_2(x,D) + \hat{q}_2) W_+}_{L^2(S_h)}^2 =  o(1), & \label{carleman_claim1} \\
 & h^3 \norm{e^{-\hphi/h} (-\Delta + \hat{A}_2(x,D) + \hat{q}_2) W_+}_{L^2(\Omega \smallsetminus S_h)}^2 & \notag \\
 & \lesssim o(1) + h (\norm{e^{-\hphi/h} W_+}^2 + \norm{e^{-\hphi/h} hDW_+}^2). & \label{carleman_claim2} 
\end{eqnarray}

\emph{Proof of \eqref{carleman_claim1}.} By \eqref{w2tilde_equation}, we have $(-\Delta + \hat{A}_2(x,D) + \hat{q}_2) W_+ = (-\Delta + \hat{A}_2(x,D) + \hat{q}_2) W_{1,+}$ in $S_h$. To obtain sufficient decay in $h$, we will need to convert $-\Delta W_{1,+}$ into first order derivatives of $W_{1,\pm}$ by noting that $W_1$ solves 
\begin{equation*}
\sigma \cdot (D+A_1) W_{1,+} + q_{1,-} W_{1,-} = 0.
\end{equation*}
Then applying $\sigma \cdot (D+A_1)$ implies 
\begin{equation*}
-\Delta W_{1,+} = -2A_1 \cdot D W_{1,+} - q_{1,-} \sigma \cdot D W_{1,-} + Q_{1,+} W_{1,+} + Q_{1,-} W_{1,-}
\end{equation*}
for some smooth matrices $Q_{1,\pm}$. Thus, the most significant terms in the expression $(-\Delta + \hat{A}_2(x,D) + \hat{q}_2) W_{1,+}$, regarding growth in $h$, are $\hat{A}_2(x,D) W_{1,+}$ and $-2A_1 \cdot D W_{1,+}$ and $- q_{1,-} \sigma \cdot D W_{1,-}$. Since $W_{1,+} = e^{\rho/h} M$ and $W_{1,-} = e^{\rho/h} R$ where $\norm{M}_{H^1(\Omega)} \lesssim 1$, $\norm{R}_{H^1(\Omega)} \lesssim h$, we have 
\begin{multline*}
h^3 \norm{e^{-\hphi/h} (-\Delta + \hat{A}_2(x,D) + \hat{q}_2) W_+}_{L^2(S_h)}^2 \\
 \lesssim h \norm{e^{-\hphi/h} e^{\rho/h} \hat{A}_2(x,hD-i\nabla \rho) M}_{L^2(S_h)}^2 + h \norm{e^{-\hphi/h} e^{\rho/h} M}_{L^2(S_h)}^2 \\ 
 + h \norm{e^{-\hphi/h} e^{\rho/h} R}_{L^2(S_h)}^2 \lesssim h^{1-C\alpha} (\abs{\log h^{\alpha}} + 1)
\end{multline*}
since $e^{-\hphi/h} \leq h^{-C\alpha} e^{-\vphi/h}$ where $C$ is independent of $\alpha$. Choosing $\alpha > 0$ so small that $1-C \alpha > 0$, this goes to zero as $h \to 0$.

\emph{Proof of \eqref{carleman_claim2}.} In $\Omega \smallsetminus S_h$ the coefficient $q_{2,-}$ is close to zero, and we will use that $\tilde{W}_2$ solves the equations 
\begin{align*}
\sigma \cdot (D+A_2) \tilde{W}_{2,+} + q_{2,-} \tilde{W}_{2,-} &= 0, \\
\sigma \cdot (D+A_2) \tilde{W}_{2,-} + q_{2,+} \tilde{W}_{2,+} &= 0,
\end{align*}
which implies that 
\begin{equation*}
(-\Delta + 2A_2 \cdot D + \tilde{q}_2) \tilde{W}_{2,+} + (\sigma \cdot Dq_{2,-}) \tilde{W}_{2,-} = 0.
\end{equation*}
By the definition of $\hat{A}_2(x,D)$ and $\hat{q}_2(x)$, we have on $\Omega \smallsetminus S_h$ 
\begin{multline} \label{claim2_intermediate}
(-\Delta + 2\hat{A}_2(x,D) + \hat{q}_2) W_{+} = (-\Delta + 2 A_2 \cdot D + \tilde{q}_2) W_{1,+} \\
 - (\sigma \cdot Dq_{2,-}) W_{-} + (\sigma \cdot Dq_{2,-}) W_{1,-}. 
\end{multline}
Note that we have written $\tilde{W}_{2,-}$ in terms of $W_-$ and $W_{1,-}$. For the first term on the right hand side of \eqref{claim2_intermediate}, a similar argument as in the proof of \eqref{carleman_claim1} implies 
\begin{equation*}
h^3 \norm{e^{-\hphi/h} (-\Delta + 2 A_2 \cdot D + \tilde{q}_2) W_{1,+}}_{L^2(\Omega)}^2 \lesssim h^{1-C\alpha} = o(1)
\end{equation*}
when $\alpha$ is small enough. Since $W_{1,-}$ has explicit form, the third term satisfies 
\begin{equation*}
h^3 \norm{e^{-\hphi/h} (\sigma \cdot Dq_{2,-}) W_{1,-}}_{L^2(\Omega)}^2 \lesssim h^{3-C\alpha} = o(1).
\end{equation*}

To prove \eqref{carleman_claim2}, it remains to show that 
\begin{multline} \label{carleman_claim2_lastestimate}
h^3 \norm{e^{-\hphi/h} (\sigma \cdot Dq_{2,-}) W_{-}}_{L^2(\Omega \smallsetminus S_h)}^2 \\
 \lesssim o(1) + h (\norm{e^{-\hphi/h} W_+}^2 + \norm{e^{-\hphi/h} hDW_+}^2).
\end{multline}
To this end we will apply the Carleman estimate for a Dirac operator given in Lemma \ref{lemma:carleman_dirac}. This will allow to estimate $W_-$ by $\sigma \cdot (D+A_2)W_-$, which again may be broken into terms involving the explicit solutions $W_{1,\pm}$ and the term $W_+$ which is admissible.

However, the Carleman estimate only applies to functions vanishing on the boundary. One has $W_{-}|_{\Gamma} = 0$ since $U_{1,-} = \tilde{U}_{2,-}$ on $\Gamma$, but $W_-$ could be nonzero on $\Gamma^c$. Here we are saved by the fact that the estimate is over the set $\Omega \smallsetminus S_h$ which has to be a positive distance away from $\Gamma^c$ if $h$ is small, by the assumption that $q_{2,-} \neq 0$ on $\Gamma^c$. Thus, let $V$ be a neighborhood of $\Gamma^c$ in which $q_{2,-} \neq 0$, choose $\chi_0 \in C^{\infty}_c(V)$ with $\chi_0 = 1$ near $\Gamma^c$, and let $\chi = 1-\chi_0$. Then 
\begin{equation*}
h^3 \norm{e^{-\hphi/h} (\sigma \cdot Dq_{2,-}) W_{-}}_{L^2(\Omega \smallsetminus S_h)}^2 \lesssim h^3 \norm{e^{-\hphi/h} \chi W_-}_{L^2(\Omega)}^2.
\end{equation*}
Applying Lemma \ref{lemma:carleman_dirac} to $\chi W_- \in H^1_0(\Omega)^{2 \times 2}$ gives 
\begin{multline*}
\norm{e^{-\hphi/h} \chi W_-}_{L^2(\Omega)}^2 \lesssim \eps(h) \norm{e^{-\hphi/h} (\sigma \cdot (D+A_2)) (\chi W_-)}_{L^2(\Omega)}^2 \\
 \lesssim \eps(h) \norm{e^{-\hphi/h} (\sigma \cdot (D+A_2)) W_-}_{L^2(\Omega)}^2 + \eps(h) \norm{e^{-\hphi/h} (\sigma \cdot D\chi) W_-}_{L^2(\Omega)}^2.
\end{multline*}
We write 
\begin{align*}
\sigma \cdot (D+A_2) W_- &= \sigma \cdot (D+A_1)W_{1,-} - \sigma \cdot (D+A_2)\tilde{W}_{2,-} \\
 & \qquad + \sigma \cdot (A_2-A_1) W_{1,-} \\
 &= -q_{1,+} W_{1,+} + q_{2,+} \tilde{W}_{2,+} + \sigma \cdot (A_2-A_1) W_{1,-} \\
 &= -q_{2,+} W_+ + (q_{2,+}-q_{1,+}) W_{1,+} + \sigma \cdot (A_2-A_1) W_{1,-}.
\end{align*}
Thus 
\begin{align*}
\norm{e^{-\hphi/h} (\sigma \cdot (D+A_2)) W_-} &\lesssim \norm{e^{-\hphi/h} W_+} + \norm{e^{-\hphi/h} W_{1,\pm}} \\
 &\lesssim \norm{e^{-\hphi/h} W_+} + h^{-C\alpha}
\end{align*}
by the explicit form of $W_1$. Finally, since $q_{2,-} \neq 0$ on the support of $\sigma \cdot D\chi$, we have in this set 
\begin{align*}
W_- &= W_{1,-} + \frac{1}{q_{2,-}} \sigma \cdot (D+A_2) \tilde{W}_{2,+} \\
 &= W_{1,-} - \frac{1}{q_{2,-}} \sigma \cdot (D+A_2) W_{+} + \frac{1}{q_{2,-}} \sigma \cdot (D+A_2) W_{1,+} \\
 &= W_{1,-} - \frac{1}{q_{2,-}} \sigma \cdot (D+A_2) W_{+} + \frac{1}{q_{2,-}} \sigma \cdot (A_2-A_1) W_{1,+} - \frac{q_{1,-}}{q_{2,-}} W_{1,-}.
\end{align*} 
Consequently 
\begin{equation*}
\norm{e^{-\hphi/h} (\sigma \cdot D\chi) W_-}_{L^2(\Omega)} \lesssim h^{-C\alpha} + \norm{e^{-\hphi/h}W_+} + \norm{e^{-\hphi/h} D W_+}.
\end{equation*}
Combining these estimates gives 
\begin{multline*}
h^3 \norm{e^{-\hphi/h} (\sigma \cdot Dq_{2,-}) W_{-}}_{L^2(\Omega \smallsetminus S_h)}^2 \\
 \lesssim h^{3-C\alpha} \eps(h) + h \eps(h) (\norm{e^{-\hphi/h} W_+}^2 + \norm{e^{-\hphi/h} hDW_+}^2).
\end{multline*}
This shows \eqref{carleman_claim2_lastestimate} if $\alpha$ is chosen small enough. The proof is complete.
\end{proof}

Next we will prove Lemma \ref{lemma:boundaryterm_electric}, which is used in recovering the electric potentials. The stronger decay of suitable blocks in the boundary integral ($o(h)$ instead of $o(1)$ as in Lemma \ref{lemma:boundaryterm_magnetic}) is due to the fact that $A_1 = A_2 = A$. Otherwise, the proof will be mostly parallel to that of Lemma \ref{lemma:boundaryterm_magnetic}.

\begin{proof}[Proof of Lemma \ref{lemma:boundaryterm_electric}]
The solutions $U_1$ and $U_2$ have the form 
\begin{align*}
U_{1,+} &= e^{\rho/h} e^{i\phi} \left[ -\frac{1}{z} r^{-1/2} a_1 \begin{pmatrix} 0 & \sigma \cdot \zeta \end{pmatrix}_{2 \times 4} + \begin{pmatrix} \hat{R}_1 & \hat{R}_1' \end{pmatrix}_{2 \times 4} \right], \\
U_{2,+}^* &= e^{-\rho/h} e^{-i\phi} \Big[ \frac{1}{z} r^{-1/2} a_2 \begin{pmatrix} 0 \\ \sigma \cdot \zeta \end{pmatrix}_{4 \times 2} + \begin{pmatrix} \hat{R}_2 \\ \hat{R}_2' \end{pmatrix}_{4 \times 2} \Big],
\end{align*}
with $\norm{\hat{R}_j}_{H^1(\Omega)} \lesssim h$, $\norm{\hat{R}_j'}_{H^1(\Omega)} \lesssim h$.

We denote the upper left $2 \times 2$ block of $(\frac{1}{q_{1,-}} \partial_{\nu} U_+|U_2^*)_{\Gamma^c}$ by $J$, and will show that $J = o(h)$. The argument for the upper right block is analogous. If $W_1$, $\tilde{W}_2$ are the left $4 \times 2$ blocks of $U_1$ and $\tilde{U}_2$, respectively, and if $W = W_1 - \tilde{W}_2$, then 
\begin{equation*}
J = \int_{\Gamma^c} e^{-\rho/h} e^{-i\phi} \hat{R}_2 \frac{1}{q_{1,-}} \partial_{\nu} W_+ \,dS.
\end{equation*}
Repeating the argument in the proof of Lemma \ref{lemma:boundaryterm_magnetic}, we obtain the estimate \eqref{j_carleman_estimate}:
\begin{multline*}
\norm{J}^2 + \frac{h}{\varepsilon(h)} (\norm{e^{-\hphi/h} W_+}^2 + \norm{e^{-\hphi/h} hDW_+}^2 ) \\
\lesssim h^3 \norm{e^{-\hphi/h} (-\Delta + \hat{A}_2(x,D) + \hat{q}_2) W_+}^2.
\end{multline*}
Here $\hat{A_2}$ and $\hat{q_2}$ are the coefficients in Lemma \ref{lemma:carleman_hdependent} for $A_2 = A$, $q_{2,-}$, and $\tilde{q}_2 = A \cdot A + D \cdot A - q_{2,+} q_{2,-}$. Let $S_h = \{ x \in \Omega \,;\, \abs{\frac{1}{q_{2,-}}} \leq \sqrt{\log h^{\alpha}} \}$ as before. Then the desired conclusion $J = o(h)$ will be a consequence of the following two estimates:
\begin{eqnarray}
 & h^3 \norm{e^{-\hphi/h} (-\Delta + \hat{A}_2(x,D) + \hat{q}_2) W_+}_{L^2(S_h)}^2 =  o(h^2), & \label{carleman_claim1_e1} \\
 & h^3 \norm{e^{-\hphi/h} (-\Delta + \hat{A}_2(x,D) + \hat{q}_2) W_+}_{L^2(\Omega \smallsetminus S_h)}^2 & \notag \\
 & \lesssim o(h^2) + h (\norm{e^{-\hphi/h} W_+}^2 + \norm{e^{-\hphi/h} hDW_+}^2). & \label{carleman_claim2_e1} 
\end{eqnarray}

\emph{Proof of \eqref{carleman_claim1_e1}.} In $S_h$ one has 
\begin{equation*}
(-\Delta + \hat{A}_2(x,D) + \hat{q}_2) W_+ = (-\Delta + 2A \cdot D - \frac{1}{q_{2,-}}(\sigma \cdot Dq_{2,-}) \sigma \cdot (D+A) + \tilde{q}_2) W_{1,+}.
\end{equation*}
Using that $W_1$ is a solution of the Dirac system with $A_1 = A$, we obtain after some computations that 
\begin{equation*}
(-\Delta + 2A \cdot D) W_{1,+} = M_{1,+} W_{1,+} + M_{1,-} W_{1,-}
\end{equation*}
in $\Omega$, where $M_{1,\pm}$ are smooth matrices in $\closure{\Omega}$ which are uniformly bounded with respect to $h$. Using the Dirac equation again, we have in $S_h$
\begin{equation*}
(-\Delta + \hat{A}_2(x,D) + \hat{q}_2) W_+ = \hat{M}_{1,+} W_{1,+} + \hat{M}_{1,-} W_{1,-},
\end{equation*}
where $\hat{M}_{1,\pm}$ are matrices in $S_h$ satisfying $\norm{\hat{M}_{1,\pm}}_{L^{\infty}(S_h)} \lesssim \sqrt{\abs{\log h^{\alpha}}}$. By the explicit form for $W_{1,\pm}$, one has 
\begin{equation*}
h^3 \norm{e^{-\hphi/h} (-\Delta + \hat{A}_2(x,D) + \hat{q}_2) W_+}_{L^2(S_h)}^2 \lesssim h^{3-C\alpha} \abs{\log h^{\alpha}}.
\end{equation*}
This proves \eqref{carleman_claim1_e1} if $\alpha$ is chosen small enough.

\emph{Proof of \eqref{carleman_claim2_e1}.} In $\Omega \smallsetminus S_h$ we obtain the identity \eqref{claim2_intermediate} where $A_2 = A$:
\begin{multline*}
(-\Delta + 2\hat{A}_2(x,D) + \hat{q}_2) W_{+} = (-\Delta + 2 A \cdot D + \tilde{q}_2) W_{1,+} \\
 - (\sigma \cdot Dq_{2,-}) W_{-} + (\sigma \cdot Dq_{2,-}) W_{1,-}. 
\end{multline*}
As in the proof of \eqref{carleman_claim1_e1}, it follows that 
\begin{equation*}
h^3 \norm{e^{-\hphi/h} (-\Delta + 2 A \cdot D + \tilde{q}_2) W_{1,+}}^2_{L^2(\Omega)} \lesssim h^{3-C\alpha} = o(h^2)
\end{equation*}
for $\alpha$ small. Also, clearly 
\begin{equation*}
h^3 \norm{e^{-\hphi/h} (\sigma \cdot Dq_{2,-}) W_{1,-}}_{L^2(\Omega)}^2 \lesssim h^{3-C\alpha} = o(h^2).
\end{equation*}
Using the cutoff $\chi$ and the Carleman estimate of Lemma \ref{lemma:carleman_dirac} in the same way as when proving \eqref{carleman_claim2}, we have 
\begin{multline*}
h^3 \norm{e^{-\hphi/h} (\sigma \cdot Dq_{2,-}) W_{-}}_{L^2(\Omega \smallsetminus S_h)}^2 \lesssim h^3 \norm{e^{-\hphi/h} \chi W_-}_{L^2(\Omega)}^2 \\
 \lesssim h^3 \eps(h) (\norm{e^{-\hphi/h} (\sigma \cdot (D+A)) W_-}_{L^2(\Omega)}^2 + \norm{e^{-\hphi/h} (\sigma \cdot D\chi) W_-}_{L^2(\Omega)}^2).
\end{multline*}
We note that 
\begin{equation*}
\sigma \cdot (D+A) W_- = -q_{2,+} W_+ + (q_{2,+}-q_{1,+}) W_{1,+} \qquad \text{in } \Omega,
\end{equation*}
and 
\begin{equation*}
W_- = W_{1,-} - \frac{1}{q_{2,-}} \sigma \cdot (D+A) W_{+} - \frac{q_{1,-}}{q_{2,-}} W_{1,-} \quad \text{on } \supp(\sigma \cdot D\chi).
\end{equation*}
This implies that 
\begin{multline*}
h^3 \norm{e^{-\hphi/h} (\sigma \cdot Dq_{2,-}) W_{-}}_{L^2(\Omega \smallsetminus S_h)}^2 \\
 \lesssim h^{3-C\alpha} \eps(h) + h \eps(h)(\norm{e^{-\hphi/h} W_+}^2 + \norm{e^{-\hphi/h} hDW_+}^2).
\end{multline*}
Now $h^{3-C\alpha} \eps(h) = o(h^2)$ for $\alpha$ small, so we have proved \eqref{carleman_claim2_e1}.
\end{proof}



\begin{thebibliography}{10}

\bibitem{astalapaivarinta} K.~Astala, L.~P\"aiv\"arinta, \textit{Calder{\'o}n's inverse conductivity problem in the plane}, Ann. of Math. \textbf{163} (2006), 265--299.

\bibitem{berthier_dirac}
A.~Boutet~de Monvel-Berthier, \emph{An optimal {C}arleman-type inequality for
  the {D}irac operator}, Stochastic processes and their applications in
  mathematics and physics (Bielefeld, 1985), Math. Appl., vol.~61, Kluwer Acad.
  Publ., 1990, pp.~71--94.

\bibitem{bukhgeimuhlmann}
A.~L. Bukhgeim and G.~Uhlmann, \emph{{Recovering a potential from partial
  Cauchy data}}, Comm. PDE \textbf{27} (2002), 653--668.

\bibitem{calderon}
A.~P. Calder{\'o}n, \emph{On an inverse boundary value problem}, Seminar on
  Numerical Analysis and its Applications to Continuum Physics, Soc. Brasileira
  de Matem{\'a}tica, R{\'i}o de Janeiro, 1980.

\bibitem{cos}
P.~Caro, P.~Ola, and M.~Salo, \emph{Inverse boundary value problem for Maxwell equations
with local data}, preprint (2009).

\bibitem{DKSaU}
D.~Dos Santos~Ferreira, C.~E. Kenig, M.~Salo, and G.~Uhlmann,
  \emph{{Limiting Carleman weights and anisotropic inverse problems}}, preprint 
  (2008), arXiv:0803.3508.

\bibitem{dksu}
D.~Dos Santos~Ferreira, C.~E. Kenig, J.~Sj{\"o}strand, and G.~Uhlmann,
  \emph{{Determining a magnetic Schr{\"o}dinger operator from partial Cauchy
  data}}, Comm. Math. Phys. \textbf{271} (2007), 467--488.

\bibitem{eskinralston_elasticity}
G.~Eskin and J.~Ralston, \emph{On the inverse boundary value problem for linear isotropic
  elasticity}, Inverse Problems \textbf{18} (2002), 907--921.

\bibitem{gotodirac}
M.~Goto, \emph{{Inverse scattering problem for Dirac operators with magnetic
  potentials at a fixed energy}}, Spectral and scattering theory and related
  topics (Japanese) (Kyoto), no. 994, 1997, pp.~1--14.

\bibitem{isakov}
V.~Isakov, \emph{On uniqueness in the inverse conductivity problem with local data}, Inverse Probl. Imaging \textbf{1} (2007), no.~1, p.~95--105.

\bibitem{isozakidirac}
H.~Isozaki, \emph{{Inverse scattering theory for Dirac operators}}, Ann. I. H.
  P. Physique Th{\'e}orique \textbf{66} (1997), 237--270.

\bibitem{iuy}
O.~Imanuvilov, G.~Uhlmann, and M.~Yamamoto, \emph{{Global uniqueness from partial Cauchy data in two dimensions}}, preprint (2008), arXiv:0810.2286.

\bibitem{jerison_dirac}
D.~Jerison, \emph{Carleman inequalities for the {D}irac and {L}aplace operators
  and unique continuation}, Adv. in Math. \textbf{62} (1986), no.~2, 118--134.

\bibitem{KSaU}
C.~E.~Kenig, M.~Salo, and G.~Uhlmann, \emph{{Inverse problems for the anisotropic Maxwell equations}}, preprint (2009).

\bibitem{ksu}
C.~E. Kenig, J.~Sj{\"o}strand, and G.~Uhlmann, \emph{{The Calder{\'o}n problem
  with partial data}}, Ann. of Math. \textbf{165} (2007), 567--591.

\bibitem{knudsensalo}
K.~Knudsen and M.~Salo, \emph{Determining nonsmooth first order terms from partial boundary measurements}, Inverse Probl. Imaging \textbf{1} (2007), 349--369.

\bibitem{kurylevlassas}
Y.~Kurylev and M.~Lassas, \emph{{Inverse problems and index formulae for Dirac operators}}, Adv. Math. (to appear).

\bibitem{lidirac}
X.~Li, \emph{{On the inverse problem for the Dirac operator}}, Inverse Problems
  \textbf{23} (2007), 919--932.

\bibitem{mandache_dirac}
N.~Mandache, \emph{Some remarks concerning unique continuation for the {D}irac
  operator}, Lett. Math. Phys. \textbf{31} (1994), no.~2, 85--92.

\bibitem{nachman2d} A.~Nachman, \textit{Global uniqueness for a two-dimensional inverse boundary value problem}, Ann. of Math. \textbf{143} (1996), 71--96.

\bibitem{nakamuratsuchida}
G.~Nakamura and T.~Tsuchida, \emph{{Uniqueness for an inverse boundary value
  problem for Dirac operators}}, Comm. PDE \textbf{25} (2000), 1327--1369.

\bibitem{nakamurauhlmann}
G.~Nakamura and G.~Uhlmann, \emph{Global uniqueness for an inverse boundary
  problem arising in elasticity}, Invent. Math. \textbf{118} (1994), 457--474.

\bibitem{nakamurauhlmannerratum}
G.~Nakamura and G.~Uhlmann, \emph{Erratum: Global uniqueness for an inverse boundary value problem
  arising in elasticity}, Invent. Math. \textbf{152} (2003), 205--207.

\bibitem{ops}
P.~Ola, L.~P\"aiv\"arinta, and E.~Somersalo, \emph{An inverse boundary value problem in electrodynamics}, Duke Math. J. \textbf{70} (1993), p.~617--653.

\bibitem{os}
P.~Ola and E.~Somersalo, \emph{{Electromagnetic inverse problems and
  generalized Sommerfeld potentials}}, SIAM J. Appl. Math. \textbf{56} (1996),
  no.~4, 1129--1145.

\bibitem{salotzou}
M.~Salo and L.~Tzou, \emph{Carleman estimates and inverse problems for Dirac operators}, Math. Ann. (to appear).

\bibitem{sylvesteruhlmann}
J.~Sylvester and G.~Uhlmann, \emph{A global uniqueness theorem for an inverse
  boundary value problem}, Ann. of Math. \textbf{125} (1987), 153--169.

\bibitem{tsuchida}
T.~Tsuchida, \emph{{An inverse boundary value problem for Dirac operators with small potentials}}, Kyushu J. Math. \textbf{52} (1998), 361--382.

\bibitem{uhlmannicm}
G.~Uhlmann, \emph{Inverse boundary value problems for partial differential equations}, {Proceedings of the International Congress of Mathematicians} (Berlin), Doc. Math., vol. III, 1998, pp.~77--86.

\bibitem{uhlmannselecta}
G.~Uhlmann, \emph{{Commentary on Calder{\'o}n's paper (29): ''On an inverse boundary value problem''}}, Selected papers of Alberto P. Calder{\'o}n, edited by A.~Bellow, C.~E. Kenig, and P.~Malliavin, AMS (2008), 623--636.

\end{thebibliography}

\providecommand{\bysame}{\leavevmode\hbox to3em{\hrulefill}\thinspace}
\providecommand{\href}[2]{#2}

\end{document}